\documentclass[11pt,reqno]{amsart}
\usepackage{enumerate}
\usepackage{enumitem}
\usepackage{mathrsfs}
\usepackage{url}
\usepackage{cite}
\usepackage{comment}
\usepackage{graphicx}
\usepackage{upgreek}
\usepackage{tikz}
\usetikzlibrary{arrows,shapes,trees,backgrounds}
\usepackage{fancyhdr}
\usepackage{amsfonts, amsmath, amssymb, amscd, amsthm, bm, cancel}
\usepackage[numbers,sort&compress]{natbib}
\usepackage[%backref=page,
linktocpage=true,colorlinks,citecolor=magenta,linkcolor=blue,urlcolor=magenta]{hyperref}
\usepackage[margin=1.1in]{geometry}

\newtheorem{prop}{Proposition}[section]
\newtheorem{thm}[prop]{Theorem}
\newtheorem{lem}[prop]{Lemma}

\newtheorem{defn}[prop]{Definition}

\theoremstyle{definition}
\newtheorem*{ack}{Acknowledgments}
\theoremstyle{remark}

\numberwithin{equation}{section}
\allowdisplaybreaks

\begin{document}

\title[Stability of Alexandrov-Fenchel type inequalities]{Stability of Alexandrov-Fenchel Type Inequalities for Nearly Spherical Sets in Space Forms}

%\author[Y. Wei]{Yong Wei}%${}^*$}
%\address{School of Mathematical Sciences, University of Science and Technology of China, Hefei 230026, P.R. China}
%\email{\href{mailto:yongwei@ustc.edu.cn}{yongwei@ustc.edu.cn}}
\author[R. Zhou]{Rong Zhou}
\address{School of Mathematical Sciences, University of Science and Technology of China, Hefei 230026, P.R. China}
\email{\href{mailto:zhourong@mail.ustc.edu.cn}{zhourong@mail.ustc.edu.cn}}
\author[T. Zhou]{Tailong Zhou}%${}^*$}
%\address{School of Mathematical Sciences, University of Science and Technology of China, Hefei 230026, P.R. China}
\email{\href{mailto:ztl20@ustc.edu.cn}{ztl20@ustc.edu.cn}}
\date{\today}
\subjclass[2020]{52A40; 53C42}
\keywords{quermassintegrals, weighted curvature integrals, nearly spherical sets, isoperimetric deficit}
%\thanks{}

\begin{abstract}
In this paper, we first derive a quantitative quermassintegral inequality for nearly spherical sets in $\mathbb{H}^{n+1}$ and $\mathbb{S}^{n+1}$, which is a generalization of the quantitative Alexandrov-Fenchel inequality proved in $\mathbb{R}^{n+1}$ \cite{VanBlargan2022QuantitativeQI}. Then we use this method to derive the stability of some geometric inequalities involving weighted curvature integrals and quermassintegrals for nearly spherical sets in $\mathbb{R}^{n+1}$ and $\mathbb{H}^{n+1}$.
\end{abstract}

\maketitle

%\tableofcontents

%-------------------------------------------------------------------------

\section{Introduction}
In this paper, we will discuss the stability of Alexandrov-Fenchel inequalities involving quermassintegrals and related weighted curvature integral inequalities for nearly spherical sets in space forms $\mathbb{N}^{n+1}(K)$ with constant curvature $K=1,-1,0$.

Let $\Omega$ be a bounded domain in $(\mathbb{N}^{n+1}(K),\overline{g})$, which is star-shaped with respect to the origin $O$. Denote $M=\partial\Omega$ and suppose that $M$ can be parametrized as $M=\{(\rho(1+u(x)),x):x\in\mathbb{S}^n\}$ in the polar coordinates, where $u:\mathbb{S}^n\to (-1,+\infty)$ is a $C^3$ function, $\rho>0$ is a constant. We call that $M$ is a nearly spherical set, if there exists a small constant $\varepsilon>0$ such that $\|u\|_{W^{2,\infty}(\mathbb{S}^n)}<\varepsilon$. 

Fuglede \cite{Fuglede1986StabilityIT,Fuglede1989StabilityIT} studied the stability of the isoperimetric inequality for nearly spherical sets in Euclidean space with $\rho=1$. Let\ $\Omega\subset\mathbb{R}^{n+1}$ be a domain with nearly spherical boundary $\partial\Omega=M=\{(1+u(x),x):x\in\mathbb{S}^n\}$, where $\|u\|_{C^1(\mathbb{S}^n)}<\varepsilon$. Under the condition 
\begin{equation}
	\mathrm{Vol}(\Omega)=\mathrm{Vol}(B),\ \mathrm{bar}(\Omega)=O,  \label{condition}
\end{equation}
\textcolor{black}{where $B$ is the unit ball centered at $O$ in $\mathbb{R}^{n+1}$}, $\mathrm{bar}(\Omega) = \dfrac{1}{\mathrm{Area}(\mathbb{S}^n)}\displaystyle\int_{\mathbb{S}^n} (1+u)^{n+2}x\mathrm{d}A$ is the barycenter of $\Omega$, he gave a lower bound for the isoperimetric deficit
\begin{equation*}
	\bar{\delta}_{0,-1}(\Omega) := \dfrac{\mathscr{A}_0(\Omega)- \mathscr{A}_0(B)}{\mathscr{A}_0(B)}
\end{equation*}
concerning $\|u\|_{L^2(\mathbb{S}^n)}^2+\|\nabla u\|_{L^2(\mathbb{S}^n)}^2$, which is vanishing if and only if $\Omega$ is the unit ball. Here $\mathscr{A}_0(\cdot)$ is the 0th quermassintegral of a domain, which measures the perimeter of the boundary (see section \ref{quermass-subsec} for the definition of the $k$th quermassintegrals $\mathscr{A}_k(\cdot),\ k=-1,\cdots,n$). In specific, there exists a constant $C$ independent of $\Omega$ such that
\begin{equation}
	\bar{\delta}_{0,-1}(\Omega)\geqslant C\|u\|_{W^{1,2}(\mathbb{S}^n)}^2, \label{Fuglede'est}
\end{equation}
and the sharp quantitative isoperimetric inequality follows.

In \cite{Cicalese2012ASP}, Cicalese and Leonardi proved a weak version of Fuglede's estimate\ (\ref{Fuglede'est}) under the same condition (\ref{condition}) as Fuglede's. They estimated the upper bound of the Fraenkel asymmetry $\bar{\alpha}(\Omega)$ in Euclidean space, where
\begin{equation}
	\bar{\alpha}(\Omega) = \inf\left\lbrace \dfrac{\mathrm{Vol}\left(\Omega\Delta B_\rho(x)\right)}{\mathrm{Vol}(B_{\rho})} :x\in\mathbb{R}^{n+1},\mathrm{Vol}(\Omega)=\mathrm{Vol}(B_{\rho})\right\rbrace \label{alphaomegaeuclidean}
\end{equation}
measures the $L^1$-distance between the set $\Omega$ and an optimal geodesic ball of the same volume. %It is obvious that $\overline{\alpha}(\Omega)=0$ if and only if $\Omega=B_{\rho}(O)$. That is to say, the smaller the value of $\overline{\alpha}(\Omega)$ is, the closer $\Omega$ is to a geodesic ball.

The quantitative isoperimetric inequality in $\mathbb{R}^{n+1}$ asks if there exists a constant $C(n)>0$ such that for all Borel sets $\Omega$ with finite measure satisfy
\begin{equation*}
	\bar{\delta}_{0,-1}(\Omega) \geqslant  C(n)\bar{\alpha}^m (\Omega)
\end{equation*}
for some exponent $m$. 

If we write $u=\sum\limits_{k=0}^{\infty}a_kY_k$, where $\{Y_k\}_{k=0}^{\infty}$ corresponds the spherical harmonics which forms an orthonormal basis for $L^2(\mathbb{S}^n)$, then \eqref{condition} implies
\begin{equation}
	a_0^2 = O(\varepsilon)\|u\|_{L^2(\mathbb{S}^n)}^2,\ 	a_1^2=O(\varepsilon)\|u\|_{L^2(\mathbb{S}^n)}^2,  \label{a1epwilon}
\end{equation}
and	consequently the isoperimetric deficit $\overline{\delta}_{0,-1}(\Omega)$ is bounded from below by $\|u\|_{L^2(\mathbb{S}^n)}^2+\|\nabla u\|_{L^2(\mathbb{S}^n)}^2$. With this observation, authors of \cite{Cicalese2012ASP} used the Fraenkel asymmetry $\overline{\alpha}(\Omega)$ to characterize the stability of the isoperimetric inequality and proved that
\begin{equation}
	\bar{\delta}_{0,-1}(\Omega) \geqslant \left( C(n) + O(\varepsilon)\right)\bar{\alpha}^2 (\Omega).
\end{equation}
The methods that Fuglede's estimate used can also be applied to prove the Minkowski inequality for nearly spherical sets in Euclidean space, interested readers can refer to Glaudo's work \cite{Glaudo2021MinkowskiIF}.
\subsection{Quantitative Alexandrov-Fenchel inequalities in space forms}
Quantitative isoperimetric inequalities for nearly spherical sets can also be considered in hyperbolic space and on the sphere. In a general space form $\mathbb{N}^{n+1}(K),\ K=-1,1$, \eqref{alphaomegaeuclidean} and (\ref{fugger}) is not invariant under scaling, so now the isoperimetric deficit $\delta_{0,-1}(\Omega)$ for $\Omega$ is defined as
\begin{equation}
	\delta_{0,-1}(\Omega) = \mathscr{A}_0(\Omega) - \mathscr{A}_0(\overline{B}_{\rho}),
\end{equation}
\textcolor{black}{where $\overline{B}_{\rho}$ is the geodesic ball with radius $\rho$ centered at $O$ in $\mathbb{N}^{n+1}(K)\ (K=-1,1)$}, and the definition for the Fraenkel asymmetry is adjusted in the following rescaled way:
\begin{defn}[\cite{Bgelein2015ASQ}]
	\label{asymmetry}
	The Fraenkel asymmetry of $\Omega\textcolor{black}{\subset\mathbb{N}^{n+1}(K)\ (K=-1,1)}$, denoted by $\alpha(\Omega)$,  is defined as
	\begin{equation}\label{F-Asym-rescale}
		\alpha(\Omega) = \inf\left\lbrace\mathrm{Vol}\left( \Omega\Delta \overline{B}_{\rho}(x)\right):x\in\mathbb{N}^{n+1}(K),\mathrm{Vol}(\Omega)=\mathrm{Vol}(\overline{B}_{\rho}(x)) \right\rbrace,
	\end{equation}
	where  $\Delta$ is the symmetric difference between two sets.
\end{defn}
In \cite{Bgelein2015ASQ,Bgelein2016AQI}, B{\"o}gelein, Duzaar, Scheven and Fusco generalized Fuglede's estimate to
\begin{equation}
	\dfrac{\delta_{0,-1}(\Omega)}{\mathscr{A}_0(\overline{B}_{\rho}))}\geqslant C\left(\dfrac{\alpha(\Omega)}{\mathscr{A}_{-1}(\overline{B}_{\rho})} \right)^2 \label{fugger}
\end{equation} 
for nearly spherical domain $\Omega\subset\mathbb{N}^{n+1}(K)$ enclosed by $M=\{(\rho(1+u(x)),x):x\in\mathbb{S}^n\}$ in $\mathbb{N}^{n+1}(K)\ (K=-1,1)$. They have defined the barycenter of a set in space forms, which is crucial to derive a technical Poincar\'e-type estimate.

\begin{defn}[The barycenter of a set in $\mathbb{N}^{n+1}(K)$ \cite{Bgelein2015ASQ,Bgelein2016AQI}]
	\label{defofbary}
	The barycenter of a set $\Omega\subset \mathbb{N}^{n+1}(K)$, denoted by $\mathrm{bar}(\Omega)$, is defined as a minimizer $p\in \mathbb{N}^{n+1}(K)$ of the function
	\begin{equation}
		p\mapsto \int_{\Omega} d_K^2(y,p)\mathrm{d}\mu_K(y).
	\end{equation}
	where $d_K(y,p)$ represents the geodesic distance between $y\in\Omega$ and $p$, $\mathrm{d}\mu_K(y)$ is the measure respect to $\mathbb{N}^{n+1}(K)$.
\end{defn}

%It has been verified in their paper \cite{Bgelein2015ASQ,Bgelein2016AQI} that for a nearly spherical domain $\Omega\subset\mathbb{N}^{n+1}(K)$ enclosed by $M=\{(\rho(1+u(x)),x):x\in\mathbb{S}^n\}$, the condition $\mathrm{bar}(\Omega)=O$ contributes to
%\begin{equation}
%	\int_0^{\rho} \int_{\mathbb{S}^n} r(1+u)^2\phi^n(r(1+u))x\mathrm{d}A\mathrm{d}r=0.
%\end{equation}ef{a1epwilon}) also holds in $\mathbb{N}^{n+1}(K)\ (K=-1,1)$. This is crucial to derive a technical Poincar\'e-type estimate.
In 1930s, Alexandrov and Fenchel proved the following inequalities for convex body $\Omega$ in $\mathbb{R}^{n+1}$:
\begin{equation}\label{af-ineq}
	\frac{\mathscr{A}_k(\Omega)}{{n\choose k}\mathrm{Area}(\mathbb{S}^n)}\leq\left(\frac{\mathscr{A}_l(\Omega)}{{n\choose l}\mathrm{Area}(\mathbb{S}^n)}\right)^{\frac{n-k}{n-l}},\ -1\leq k<l\leq n.
\end{equation}
Equality holds if and only if $\Omega$ is a ball. When $l=-1,\ k=0$, \eqref{af-ineq} reduces to the classic isoperimetric inequality. Guan and Li \cite{GL09} proved the Alexandrov-Fenchel inequalities \eqref{af-ineq} for all star-shaped and $k$-convex domains in $\mathbb{R}^{n+1}$, using curvature flow method. It is conjectured that for domain $\Omega$ in space form $\mathbb{N}^{n+1}(K),\ K=\pm 1,$ satisfying suitable convexity, the Alexandrov-Fenchel inequalities take the form of (see for instance in \cite{B-G-L})
\begin{equation}\label{af-ineq-sf}
	\mathscr{A}_k(\Omega)\leq\left(\psi_k\circ\psi_l^{-1}\right)\left(\mathscr{A}_l(\Omega)\right),\ -1\leq k<l\leq n,
\end{equation}
where $\psi_m(\rho):=\mathscr{A}_m(\bar{B}_\rho)$.

Our work is motivated by VanBlargan and Wang \cite{VanBlargan2022QuantitativeQI}, they proved the quantitative Alexandrov-Fenchel inequalities for nearly spherical sets in $\mathbb{R}^{n+1}$. For $0\leqslant j<k$, $1\leqslant k\leqslant n-1$, let $\Omega$ be enclosed by a nearly spherical set $M=\{(1+u(x),x):x\in\mathbb{S}^n\}$. Under the condition
\begin{equation}
	\mathscr{A}_j(\Omega) = \mathscr{A}_j(B),\ \ \ \mathrm{bar}(\Omega)=O, \label{conditionquer}
\end{equation}
they proved
%\begin{equation}
%		\|\nabla u\|_{L^2(\mathbb{S}^n)}^2 \geqslant 2(n+1)\|u\|_{L^2(\mathbb{S}^n)}^2 + O(\varepsilon)\|u\|_{L^2(\mathbb{S}^n)}^2 + O(\varepsilon)\|u\|_{L^2(\mathbb{S}^n)}^2
%\end{equation}
\begin{equation}\label{aklowebound}
\mathscr{A}_k(\Omega) - \mathscr{A}_k(B)\geqslant C(n,k,j) \left(  (1+O(\varepsilon))\|u\|_{L^2(\mathbb{S}^n)}^2 + \left(\dfrac{1}{2}+O(\varepsilon)\right)\|\nabla u\|_{L^2(\mathbb{S}^n)}^2  \right),
\end{equation}
where $C(n,k,j)=\displaystyle {n\choose k} \dfrac{(n-k)(k-j)}{2n}$. Then the $(k,j)$-isoperimetric deficit
\begin{equation}
	\bar{\delta}_{k,j}(\Omega) = \dfrac{\mathscr{A}_k(\Omega)-\mathscr{A}_k(B_{\Omega,j})}{\mathscr{A}_k(B_{\Omega,j})} , \label{kjdeficit}
\end{equation}
where $B_{\Omega,j}$ is the ball centered at $O$ such that $\mathscr{A}_j(\Omega)=\mathscr{A}_j(B_{\Omega,j})$, is bounded form below by the Fraenkel asymmetry:
\begin{equation}
	\bar{\delta}_{k,j}(\Omega) \geqslant \left( \dfrac{n(n-k)(k-j)}{4(n+1)^2}+O(\varepsilon) \right)\overline{\alpha}^2(\Omega).\label{deltakjresult}
\end{equation}

In general space forms, we also consider the rescaled isoperimetric deficit in contrast to (\ref{kjdeficit}) in Euclidean space.
\begin{defn}[The $(k,j)$-isoperimetric deficit]
	Let $\Omega$ be a domain in $\mathbb{N}^{n+1}(K)$. For any given $0\leqslant k\leqslant n$, $-1\leqslant j<k$, define the $(k,j)$-isoperimetric deficit for $\Omega$, denoted by $\delta_{k,j}(\Omega)$, as
	\begin{equation}
		\delta_{k,j}(\Omega) = \mathscr{A}_k(\Omega) - \mathscr{A}_k(\overline{B}_{\Omega,j}),
	\end{equation}
	where $\overline{B}_{\Omega,j}$ is a geodesic ball centered at $O$ in $\mathbb{N}^{n+1}(K)$ which satisfies
	$\mathscr{A}_j(\Omega) = \mathscr{A}_j(\overline{B}_{\Omega,j})$.
\end{defn}

The first goal of this paper is to give the quantitative version of Alexandrov-Fenchel inequalities \eqref{af-ineq-sf} in space forms under the same assumption (\ref{conditionquer}) as in \cite{VanBlargan2022QuantitativeQI,Bgelein2015ASQ,Bgelein2016AQI}, as a generalization of (\ref{deltakjresult}): For $\Omega$ enclosed by nearly spherical sets $M$ in hyperbolic space $\mathbb{H}^{n+1}$ and on the sphere $\mathbb{S}^{n+1}$,
\begin{equation}
	\delta_{k,j}(\Omega)\geqslant C\alpha^2(\Omega),
\end{equation}
where\ $C$ is a constant independent of $\Omega$. We make an effort to derive the expression for $\mathscr{A}_k(\Omega)-\mathscr{A}_k(\overline{B}_{\rho})\ (-1\leqslant k\leqslant n)$. Our result is as follows:
\begin{thm}\label{mainresult}
	Let $M=\{(\rho(1+u(x)),x):x\in\mathbb{S}^n\}$ be a nearly spherical set in $\mathbb{N}^{n+1}(K)\ (K=-1,1)$, %$\Omega=\{ (r(1+u(x)),x):r\in[0,\rho],\ x\in\mathbb{S}^n \}\subset \mathbb{N}^{n+1}(K)\ (K=-1,1)$ be a bounded subset enclosed by $M$,
	where $u\in C^3(\mathbb{S}^n)$. $\Omega$ is a bounded subset which is star-shaped with respect to $O$ and is enclosed by $M$. Suppose $0\leqslant k\leqslant n-1$, $-1\leqslant j<k$, if both of the following hold
	\begin{enumerate}[label={(\arabic*)},parsep=0cm,itemsep=0cm,topsep=0cm]
		\item $\mathscr{A}_j(\Omega) = \mathscr{A}_j(\overline{B}_{\rho})$,
		\item $\mathrm{bar}(\Omega)=O$,
	\end{enumerate}
	then for any $\eta>0$, there exists $\varepsilon>0$, such that when $\|u\|_{W^{2,\infty}(\mathbb{S}^n)}<\varepsilon$, we have
	\begin{equation}
		\delta_{k,j}(\Omega) \geqslant \left( \dfrac{n(n-k)(k-j)}{4\mathrm{Area}(\mathbb{S}^n)}{n\choose k} \dfrac{\phi'^k(\rho)}{\phi^{n+k+2}(\rho)} -\eta \right)\alpha^2(\Omega),
	\end{equation}
	where $\phi$ is defined in \eqref{eq-phi}.
\end{thm}
\subsection{Quantitative weighted quermassintegral inequalities}
For a smooth, star-shaped and $k$-convex (i.e. $\sigma_i(\kappa)\geqslant 0\ (1\leqslant i\leqslant k)$) hypersurface in  $\mathbb{R}^{n+1}$, geometric inequalities involving weighted curvature integrals, with $\Phi$ (a radial function defined in (\ref{Phi})) as the weighting factor, have been studied in \cite{Kwong-Miao15,Girao-Rodrigues,Wei2022NewWG}. More precisely, if $M$ is such a hypersurface enclosing a bounded domain $\Omega$, then 
\begin{equation}
	\int_M \Phi\sigma_k(\kappa) \mathrm{d}\mu_g \geqslant C(n,l,k) \mathscr{A}_l(\Omega)^{\frac{n+2-k}{n-l}},
    \label{weightedinR}
\end{equation}
where $1\leqslant k\leqslant n$, $-1\leqslant l<k$, and equality holds if and only if $\Omega$ is a ball. Moreover, if $M$ is h-convex (i.e. $\kappa_i\geqslant 1\ (1\leqslant i\leqslant n)$) in $\mathbb{H}^{n+1}$, Wei and the second author \cite{Wei2022NewWG} proved
\begin{equation}
	\label{weightedH}
	\int_M \Phi\sigma_k(\kappa)\mathrm{d}\mu_g \geqslant \xi_{k,l}(\mathscr{A}_{l}(\Omega)),
\end{equation}
where $\xi_{k,l}$ is the unique function such that equality holds in (\ref{weightedH}) when $\Omega$ is a geodesic ball, $1\leqslant k\leqslant n$, $-1\leqslant l<k$ and equality holds if and only if $\Omega$ is a geodesic ball. 

Besides $\Phi$, another interesting weighting factor is $\phi'$ defined in (\ref{eq-phi}). In \cite{Locallyconstrained}, the authors compared the following weighted integral of $k$th mean curvature of a smooth h-convex hypersurface to its $l$th quermassintegral: 
\begin{equation}
	\label{weightedH1}
	\int_M \phi'\sigma_k(\kappa)\mathrm{d}\mu_g \geqslant \eta_{k,l}(\mathscr{A}_{l}(\Omega)),
\end{equation}
where $\eta_{k,l}$ is the unique function such that equality holds in (\ref{weightedH1}) when $\Omega$ is a geodesic ball, $1\leqslant k\leqslant n$, $-1\leqslant l<k$ and equality holds if and only if $\Omega$ is a geodesic ball. For $k=1$, the weighted total mean curvature integral is strongly related to the quasilocal mass and the Riemannian Penrose inequality (see for instance in \cite{B-H-W}), and the corresponding weighted Minkowski inequality was proved in \cite{Lima-Girao} for strictly mean convex and star-shaped hypersurfaces. Relevant studies for the inequality \eqref{weightedH1} can be found in \cite{Ge-Wang-Wu,B-G-L}.

Our next goal is to derive the quantitative version of inequalities (\ref{weightedinR}) and (\ref{weightedH}) for nearly spherical sets $M$ in $\mathbb{R}^{n+1}$ and $\mathbb{H}^{n+1}$ characterized by $\alpha^2(\Omega)$ respectively and obtain the quantitative version of (\ref{weightedH1}) for nearly spherical sets in $\mathbb{H}^{n+1}$. In specific, it states that for any fixed $0\leqslant k\leqslant n$, if $\mathscr{A}_j(\Omega)=\mathscr{A}_j(\overline{B}_{\rho})\ (-1\leqslant j<k)$, then
\begin{equation*}
	\int_M \Psi(r)\sigma_k(\kappa)\mathrm{d}\mu_g - \int_{\partial \overline{B}_{\rho}}\Psi(r)\sigma_k(\kappa) \mathrm{d}\mu_g \geqslant C\alpha^2(\Omega),
\end{equation*}
where $\Psi=\Phi\ \mbox{or}\ \phi'$, $\Omega$ is enclosed by $M$ and $C$ is a constant independent of $\Omega$.

In the following, we present the stability of (\ref{weightedinR}) for nearly spherical sets in $\mathbb{R}^{n+1}$. Note that in this case, the weighted curvature integrals are invariant under rescaling, we only consider the stability in a domain $\Omega$ which is close to the unit ball $B$ and characterize the stability by $\overline{\alpha}^2(\Omega)$ defined in (\ref{alphaomegaeuclidean}).
\begin{thm}
	\label{mainresult2}
	Let $M=\{(1+u(x),x):x\in \mathbb{S}^n \}$ be a nearly spherical set in $\mathbb{R}^{n+1}$, where $u\in C^3(\mathbb{S}^n)$. $\Omega\subset\mathbb{R}^{n+1}$ is a bounded subset which is star-shaped with respect to $O$ and is enclosed by $M$. Suppose $0\leqslant k\leqslant n$, $-1\leqslant j<k$, if both of the following hold
	\begin{enumerate}[label={(\arabic*)},parsep=0cm,itemsep=0cm,topsep=0cm]
		\item $\mathscr{A}_j(\Omega) = \mathscr{A}_j(B)$,
		\item $\mathrm{bar}(\Omega)=O$,
	\end{enumerate}
	then for any $\eta>0$, there exists $\varepsilon>0$, such that when $\|u\|_{W^{2,\infty}(\mathbb{S}^n)}<\varepsilon$ holds, we have
	\begin{equation}
		\dfrac{\displaystyle\int_M \Phi\sigma_k(\kappa) \mathrm{d}\mu_g - \displaystyle\int_{\partial B} \Phi\sigma_k(\kappa) \mathrm{d}\mu_g}{\displaystyle\int_{\partial B} \Phi\sigma_k(\kappa) \mathrm{d}\mu_g} \geqslant \left[ \dfrac{n\left((n-k+2)(k-j)+2k-2\right)}{4(n+1)^2} -\eta \right]\overline{\alpha}^2(\Omega), \label{weiinR}
	\end{equation}
	where $\Phi$ is defined in \eqref{Phi}.
\end{thm}

We also derive the stability of inequalities (\ref{weightedH}) and (\ref{weightedH1}) for nearly spherical sets in $\mathbb{H}^{n+1}$.
\begin{thm}
	\label{mainresult3}
	Let $M=\{(\rho(1+u(x)),x):x\in\mathbb{S}^n\}$ be a nearly spherical set in $\mathbb{H}^{n+1}$, where $u\in C^3(\mathbb{S}^n)$. $\Omega\subset\mathbb{H}^{n+1}$ is a bounded subset which is star-shaped with respect to $O$ and is enclosed by $M$. Suppose $0\leqslant k\leqslant n-1$, $-1\leqslant j<k$, if both of the following hold
	\begin{enumerate}[label={(\arabic*)},parsep=0cm,itemsep=0cm,topsep=0cm]
		\item $\mathscr{A}_j(\Omega) = \mathscr{A}_j(\overline{B}_{\rho})$,
		\item $\mathrm{bar}(\Omega)=O$,
	\end{enumerate}
	then for any $\eta>0$, there exists $\varepsilon>0$, such that when $\|u\|_{W^{2,\infty}(\mathbb{S}^n)}<\varepsilon$ holds, we have
	\begin{eqnarray}
		&&\displaystyle\int_M \Phi\sigma_k(\kappa) \mathrm{d}\mu_g - \displaystyle\int_{\partial\overline{B}_{\rho}} \Phi\sigma_k(\kappa) \mathrm{d}\mu_g \notag\\
		&\geqslant& \left(\dfrac{n(n-k)(k-j)}{4\mathrm{Area}(\mathbb{S}^n)}{n\choose k}\dfrac{\phi'^{k-2}(\rho)\Phi(\rho)}{\phi^{n+k+2}(\rho)} -\eta \right)\alpha^2(\Omega),   \label{hqq}
	\end{eqnarray}
    and
	\begin{eqnarray}
		&&\int_M \phi'\sigma_k(\kappa) \mathrm{d}\mu_g - \int_{\partial \overline{B}_{\rho}} \phi'\sigma_k(\kappa) \mathrm{d}\mu_g \notag \\
		&\geqslant& \left( \dfrac{n(n-k)(k-j)}{4\mathrm{Area}(\mathbb{S}^n)}{n\choose k} \dfrac{\phi'^{k+1}(\rho)}{\phi^{n+k+2}(\rho)} -\eta \right)\alpha^2(\Omega), \label{hqq2}
	\end{eqnarray}
	where $\Phi(r)$ and $\phi(r)$ are defined in \eqref{eq-phi} and \eqref{Phi}.
\end{thm}

We should mention that the curvature flow methods could be used to establish the quantitative quermassintegral inequalities in $\mathbb{N}^{n+1}(K)$. Using the inverse curvature flow, VanBlargan and Wang \cite{VanBlargan2022StabilityOQ} managed to give the inequality (\ref{aklowebound}) a new proof; Scheuer \cite{Scheuer2023Stability} established a quantitative quermassintegral inequality for closed and star-shaped $C^2$-hypersurfaces in $\mathbb{N}^{n+1}(K)$. For the study of other weighted geometric inequalities in quantitative form, we refer to \cite{FUSCO2023109946,Gavitone}.

The paper is organized as follows. In Section \ref{section2}, we present some preliminaries for nearly spherical sets in space forms and some useful identities about symmetric polynomials. In Section \ref{section3}, we derive an explicit expression for $\delta_{k,j}(\Omega)$ and weighted curvature integrals $\int_M \Phi\sigma_k(\kappa)\mathrm{d}\mu_g\ (0\leqslant k\leqslant n)$ and $\int_M \phi'\sigma_k(\kappa)\mathrm{d}\mu_g\ (0\leqslant k\leqslant n)$ under the condition $\|u\|_{W^{2,\infty}(\mathbb{S}^n)}<\varepsilon$. In Section \ref{section4}, we study the quantitative quermassintegral inequality and prove the first main result Theorem \ref{mainresult}. In Section \ref{section5}, we discuss the weighted quermassintegral inequalities and prove Theorem \ref{mainresult2} and Theorem \ref{mainresult3}.

\section{Preliminaries}
\label{section2}
Here we collect some properties of nearly spherical sets parametrized by radial function, elementary symmetric polynomials and quermassintegrals for nearly spherical sets.

\subsection{Nearly spherical sets in space forms}
Let $(\mathbb{N}^{n+1}(K),\overline{g})\ (n\geqslant 2)$ be a space form of dimension $n+1$ with sectional curvature $K=0,1,-1$. Its warped product metric is $\overline{g} = \mathrm{d}r^2 + \phi^2(r)g_{\mathbb{S}^n}$, where $r$ is the radial distance to the origin and
\begin{equation}\label{eq-phi}
	\phi(r) = \left\lbrace 
	\begin{array}{lll}
		r,&r\in[0,+\infty) ,&K=0,\\
		\sin r,&r\in[0,\dfrac{\pi}{2}) ,& K=1,\\
		\sinh r,&r\in[0,+\infty) ,&K=-1.
	\end{array}
	\right. 
\end{equation}
%Denote $\phi'(r) = \dfrac{\mathrm{d}(\phi(r))}{\mathrm{d}r}$, $\phi'(t)=\phi'(r)|_{r=t}$, $\phi''(r) = \dfrac{\mathrm{d}^2(\phi(r))}{\mathrm{d}r^2}$, 
Then we have
\begin{equation}
	(\phi')^2+K\phi^2 = 1,\ \phi''=-K\phi.  \label{hd}
\end{equation}
Define the radial function $\Phi(r)=\displaystyle\int_0^{r} \phi(s) \mathrm{d}s$ on $\mathbb{N}^{n+1}(K)$, that is
\begin{equation}\label{Phi}
	\Phi(r) = \left\lbrace 
	\begin{array}{lll}
		\dfrac{1}{2}r^2,&r\in[0,+\infty) ,&K=0,\\
		1-\cos r,&r\in[0,\dfrac{\pi}{2}) ,& K=1,\\
		\cosh r-1,&r\in[0,+\infty) ,&K=-1.
	\end{array}
	\right. 
\end{equation}
It's well known that $V=\nabla^{\mathbb{N}^{n+1}(K)}\Phi(r)=\phi(r)\dfrac{\partial}{\partial r}$ is a conformal Killing vector field. We refer to \cite{Guan2013AMC} for more details.

Let $\Omega$ be a bounded domain in $\mathbb{N}^{n+1}(K)\ (K=-1,1)$ which is star-shaped respect to the origin $O$ and is enclosed by $M$. Suppose that $M=\{(\rho(1+u(x)),x):x\in\mathbb{S}^n\}$, where $u:\mathbb{S}^n\to (-1,+\infty)$ is a $C^3$ function satisfying $\|u\|_{W^{2,\infty}(\mathbb{S}^n)}<\varepsilon$. 

We present some geometric quantities for the hypersurface $M$ under this radial parametrization. In a geodesic polar coordinate of $\mathbb{N}^{n+1}(K)$, denote $\left\{\frac{\partial}{\partial\theta_1},\frac{\partial}{\partial\theta_2},\cdots,\frac{\partial}{\partial\theta_n},\frac{\partial}{\partial r}\right\}$ as the tangent basis and $s_{ij}$ as the canonical metric on the unit sphere\ $\mathbb{S}^n$. Then
\begin{equation}
	\langle \frac{\partial}{\partial\theta_i},\frac{\partial}{\partial r} \rangle = 0,\ \ \langle \frac{\partial}{\partial r},\frac{\partial}{\partial r} \rangle = 1, \ \ \langle \frac{\partial}{\partial\theta_i},\frac{\partial}{\partial \theta_j} \rangle = \phi^2(r) s_{ij}.
\end{equation}
We also refer\ $u_i=\dfrac{\partial u}{\partial \theta_i}$, then\ $\{e_i=\dfrac{\partial}{\partial\theta_i}+\rho u_i\dfrac{\partial}{\partial r}:i=1,2,\cdots,n\}$ form a tangent basis of $M$. Let\ $g_{ij}$, $g^{ij}$, $N$, $h_{ij}$ denote the induced metric, the inverse metric matrix, the outward unit normal vector and the second fundamental form corresponding to $M$ respectively. For convenience, we denote
\begin{equation}\label{phi}
	\phi=\phi(r),\ \phi'=\phi'(r),\ r=\rho(1+u),
\end{equation}
and 
\begin{equation}\label{D}
	D=\sqrt{\phi^2+\rho^2|\nabla u|^2},
\end{equation}
where $\nabla$ is the Levi-Civita connection on $\mathbb{S}^n$. Then the induced metric of $M$ is
\begin{equation}
	g_{ij} =\rho^2 u_iu_j+\phi^2 s_{ij}.
\end{equation}
Thus the area element $\mathrm{d}\mu_g$ corresponding to the induced metric $g$ is
\begin{equation}\label{area-elem}
	\mathrm{d}\mu_g=\sqrt{\mathrm{det}(g_{ij})}\mathrm{d}A = \phi^{n-1} D\mathrm{d}A,\\
\end{equation}
the inverse of $(g_{ij})$ is
\begin{equation}
	g^{ij} = \dfrac{s^{ij}}{\phi^2}-\dfrac{1}{\phi^2}\cdot\dfrac{\rho^2 u_ku_ls^{ik}s^{jl}}{D^2},
\end{equation}
and the outward unit normal vector corresponding to $M$ is
\begin{equation}
	N = \dfrac{-\rho s^{ij}u_i\frac{\partial}{\partial\theta_j}+\phi^2\frac{\partial}{\partial r}}{\phi D}.
\end{equation}
Note that $h_{ij}=-\langle \nabla^{\mathbb{N}^{n+1}(K)}_{e_i}e_j, N\rangle$, the second fundamental form of $M$ is
\begin{equation}
	h_{ij} = \dfrac{1}{D}\left[ 2\phi'\rho^2 u_iu_j+\phi^2\phi's_{ij} - \phi \rho u_{ij} \right],
\end{equation}
then the weingarten tensor $h^i_j=g^{ik}h_{kj}$ is
\begin{equation}
	h^i_j = \dfrac{\phi'\delta^i_j}{D} - \dfrac{\rho u^i_j}{D\phi} + \dfrac{\phi'\rho^2 u^iu_j}{D^3} + \dfrac{\rho^3 u^iu_k u^k_j}{D^3\phi},
\end{equation}
where $|\nabla u|^2 = s^{ij}u_iu_j$. We also refer $u^i=s^{ij}u_j$, $u^k_j=s^{ki}u_{ij}$.

%We mention here a formula
%\begin{equation}
%	\mathrm{Vol}(\Omega\Delta \overline{B}_{\rho} ) = \int_{\mathbb{S}^n} \left|  \int_0^{\rho(1+u)} \phi^n(r)\mathrm{d}r - \int_0^{\rho} \phi^n(r)\mathrm{d}r \right| \mathrm{d}A.
%\end{equation}

\subsection{Properties of elementary symmetric functions}
\label{propertyofelement}
Here we present some properties of elementary symmetric polynomials (see e.g., in \cite{Guan2014}).

Let $\lambda=(\lambda_1,\lambda_2,\cdots,\lambda_n)\in\mathbb{R}^n$, we denote
\begin{equation}
	\sigma_k(\lambda) = \sum\limits_{1\leqslant i_1<i_2<\cdots<i_k\leqslant n} \lambda_{i_1}\lambda_{i_2}\cdots\lambda_{i_k}
\end{equation}
as the $k$th elementary symmetric polynomials of $\lambda\in\mathbb{R}^n$ when $1\leqslant k\leqslant n$ and have the convention that $\sigma_0(\lambda)=1$. Generally, let $A=(A^i_j)_{n\times n}$ be a symmetric matrix, then for all $1\leqslant k\leqslant n$,
\begin{equation}
	\sigma_k(A)= \dfrac{1}{k!}\delta_{i_1i_2\cdots i_k}^{j_1j_2\cdots j_k}A^{i_1}_{j_1}A^{i_2}_{j_2}\cdots A^{i_k}_{j_k}.  \label{ska}
\end{equation}
Besides, we set $\sigma_0(A)=1$.

Let $A_1,\cdots,A_k$ be $n\times n$ symmetric matrices, the $k$th Newton operator $(1\leqslant k\leqslant n)$ for them is defined as
\begin{equation}
	[T_k]_i^j(A_1,A_2,\cdots,A_k)=\dfrac{1}{k!}\delta_{ii_1i_2\cdots i_k}^{jj_1j_2\cdots j_k}(A_1)^{i_1}_{j_1}(A_2)^{i_2}_{j_2}\cdots (A_k)^{i_k}_{j_k}.
\end{equation}
We also have the convention that $[T_0]_i^j =\delta_i^j$. Particularly, when $A_1=A_2=\cdots=A_k=A$, we briefly denote
\begin{equation}
	[T_k]_i^j(A) = [T_k]_i^j (\underbrace{A,A,\cdots,A}_k).
\end{equation}

Note that $[T_k]_i^j(A) = \dfrac{\partial \sigma_k(A)}{\partial A^i_j}$, it is well-known that 
\begin{equation}\label{newton}
	A^j_s [T_m]^i_j(A) = \delta^i_s\sigma_{m+1}(A)-[T_{m+1}]^i_s(A).
\end{equation}	

By using the anti-symmetry property of Kronecker-Delta $\delta_{i_1i_2\cdots i_k}^{j_1j_2\cdots j_k}$, we have the following Lemma:

\begin{lem}[\cite{VanBlargan2022QuantitativeQI}, Lemma 3.2]
	Suppose that $w,v_1,v_2$ are column vectors in $\mathbb{R}^n$, then there holds
	\begin{equation}
		\dfrac{1}{(k-1)!}\delta_{i_1i_2\cdots i_k}^{j_1j_2\cdots j_k} (wv_1^t)^{i_1}_{j_1}(wv_2^t)^{i_2}_{j_2}(A_1)^{i_3}_{j_3}\cdots (A_{k-2})^{i_k}_{j_k} = 0. \label{Sigmak}
	\end{equation}
	
\end{lem}

\subsection{Quermassintegrals and weighted curvature integrals of nearly spherical sets}
\label{quermass-subsec}

In a space form $\mathbb{N}^{n+1}(K)$, the quermassintegrals of a compact convex domain $\Omega$ are defined as (see \cite{G-Solanes05}):
\begin{equation}
	\mathscr{A}_k(\Omega)=(n-k){n\choose k}\frac{\omega_k\cdots\omega_0}{\omega_{n-1}\cdots\omega_{n-k-1}}\int_{\mathcal{L}_{k+1}} \chi(L_{k+1}\cap\Omega)\mathrm{d}L_{k+1}
\end{equation}
for $k=0,1,\cdots,n-1$, where $\omega_k=|\mathbb{S}^k|$ denotes the area of $k$-dimensional unit sphere, $\mathcal{L}_{k+1}$ is the space of $(k+1)$-dimensional totally geodesic subspaces $L_{k+1}$ in $\mathbb{N}^{n+1}(K)$, the funtion $\chi$ is defined to be $1$ if $L_{k+1}\cap\Omega\neq\varnothing$ and to be $0$ otherwise. In particular, we have 
\begin{equation}
	\mathscr{A}_{-1}(\Omega)=\mathrm{Vol}(\Omega),\ \mathscr{A}_{0}(\Omega)=|\partial\Omega|,\ \mathscr{A}_{n}(\Omega)=\frac{\omega_n}{n+1}.
\end{equation}
For a domain $\Omega$ with $C^2$ boundary, we denote\ $\kappa=(\kappa_1,\kappa_2,\cdots,\kappa_n)$ as the principal curvature vector of the hypersurface $M$. Let $\mathrm{d}\mu_g$ the area element of $M$, $\mathrm{d}A$ the area element of\ $\mathbb{S}^n$, then the quermassintegrals can be defined (see \cite{Li2021IsoperimetricTI}) and calculated as follows:
\begin{eqnarray}
	\mathscr{A}_{-1}(\Omega) &=& \mathrm{Vol}(\Omega) = \displaystyle\int_{\mathbb{S}^n} \left(\displaystyle\int_0^{\rho(1+u)} \phi^n(r) \mathrm{d}r \right) \mathrm{d}A, \label{a-1omega}\\
	\mathscr{A}_0(\Omega)&=& \displaystyle\int_M 1 \mathrm{d}\mu_g = \displaystyle\int_{\mathbb{S}^n}\phi^{n-1}(\rho(1+u)) \sqrt{\phi^2(\rho(1+u))+\rho^2|\nabla u|^2} \mathrm{d}A, \label{a0omega}\\
	\mathscr{A}_1(\Omega) &=& \displaystyle\int_M \sigma_1(\kappa) \mathrm{d}\mu_g + K n\mathrm{Vol}(\Omega),\label{a1omega}\\
	\mathscr{A}_k(\Omega) &=& \displaystyle\int_M \sigma_k(\kappa) \mathrm{d}\mu_g + K \dfrac{n-k+1}{k-1}\mathscr{A}_{k-2}(\Omega)\ (2\leqslant k\leqslant n).  \label{akomega}
\end{eqnarray}

%For $\overline{B}_{\rho}$, the geodesic ball with radius $\rho$ centered at $O$ in\ $\mathbb{N}^{n+1}(K)$, the principal curvatures of\ $\partial\overline{B}_{\rho}$ are\ $\kappa_1=\kappa_2=\cdots =\kappa_n = \dfrac{\phi'(\rho)}{\phi(\rho)}$. Consequently, the quermassintegrals for $\overline{B}_{\rho}$ are
%\begin{eqnarray}
%	\mathscr{A}_{-1}(\overline{B}_{\rho}) &=& \int_{\mathbb{S}^n}\left( \int_0^{\rho} \phi^n (r)\mathrm{d}r \right) \mathrm{d}A, \label{a-1b}\\
%	\mathscr{A}_0(\overline{B}_{\rho}) &=& \phi^n(\rho) \mathrm{Area}(\mathbb{S}^n),  \label{a0b}\\
%	\mathscr{A}_1(\overline{B}_{\rho}) &=& \int_{\mathbb{S}^n} n\phi'(\rho)\phi^{n-1}(\rho)\mathrm{d}A +K n\mathrm{Vol}(\overline{B}_{\rho}),  \label{a1b}\\
%	\mathscr{A}_k(\overline{B}_{\rho}) &=&  \int_{\mathbb{S}^n} {n\choose k} \phi'^k(\rho)\phi^{n-k}(\rho) \mathrm{d}A +K \dfrac{n-k+1}{k-1} \mathscr{A}_{k-2}(\overline{B}_{\rho})\ (2\leqslant k\leqslant n). \label{akb}
%\end{eqnarray}

The weighted curvature integrals for $\Omega$ enclosed by a nearly spherical set $M=\{(\rho(1+u(x)),x):x\in\mathbb{S}^n\}$ is defined as $\displaystyle\int_{M} \Psi(r) \sigma_k(\kappa)  \mathrm{d}\mu_g\ (0\leqslant k\leqslant n)$, where $\Psi$ equals $\Phi\ \mbox{or}\ \phi'$ are defined in (\ref{eq-phi}) and (\ref{Phi}). %Then for $\overline{B}_{\rho}$, the weighted curvature integrals are
%\begin{equation}
%	\int_{\mathbb{S}^n} {n\choose k} \phi'^k(\rho) \phi^{n-k}(\rho) \Psi(\rho) \mathrm{d}A\ \ (0\leqslant k\leqslant n). \label{brphi}
%\end{equation}

\section{Derive the $(k,m)$-isoperimetric deficit}
\label{section3}
In this section, we aim to derive the specific expression for $\mathscr{A}_k(\Omega)-\mathscr{A}_k(\overline{B}_{\rho})\ (-1\leqslant k\leqslant n)$ in terms of $u$ and its derivatives. Firstly, we use (\ref{ska}) to calculate $\sigma_k(\kappa)\ (1\leqslant k\leqslant n)$ of $M$. Then, under the condition $\|u\|_{W^{2,\infty}(\mathbb{S}^n)}<\varepsilon$, we will compute the curvature integrals $\displaystyle\int_M \sigma_k(\kappa)\mathrm{d}\mu_g\ (0\leqslant k\leqslant n)$ and the weighted curvature integrals $\displaystyle\int_M \Psi\sigma_k(\kappa)\mathrm{d}\mu_g\ (0\leqslant k\leqslant n)$. Finally, one can use iteration process to obtain the asymptotic expression for $\mathscr{A}_k(\Omega)-\mathscr{A}_k(\overline{B}_{\rho})\ (-1\leqslant k\leqslant n)$.

\subsection{Computation of $\sigma_k(\kappa)$}
Using the properties of elementary symmetric polynomials that have been discussed in Section \ref{propertyofelement}, we can calculate $\sigma_k(\kappa)\ (1\leqslant k\leqslant n)$ as follows:
\begin{lem}[Expression for $\sigma_k(\kappa)\ (1\leqslant k\leqslant n)$]
	Let\ $\Omega\subset\mathbb{N}^{n+1}(K)\ (K=-1,1)$, $M=\partial\Omega=\{(\rho(1+u(x)),x):x\in\mathbb{S}^n\}$, where $u\in C^2(\mathbb{S}^n)$. Denote 
	$$\phi:=\phi(\rho(1+u)),\ \phi':=\phi'(\rho(1+u)),\ D=\sqrt{\phi^2+\rho^2|\nabla u|^2},$$ then for any\ $1\leqslant k\leqslant n$, there holds
	\begin{equation}
		\sigma_k(\kappa) = \sum\limits_{m=0}^k \dfrac{(-1)^m \phi'^{k-m}}{D^{k+2}\phi^m }{n-m \choose k-m} \rho^m \left[ \phi^2\sigma_m(D^2u) + \dfrac{k+n-2m}{n-m} \rho^2 u^iu_j[T_m]_i^j(D^2 u) \right].   \label{sigmak}
	\end{equation}
\end{lem}

\begin{proof}
	Note that
	\begin{eqnarray}
		\sigma_k(h^i_j) &=& \dfrac{1}{k!}\delta^{j_1j_2\cdots j_k}_{i_1i_2\cdots i_k}\left( \dfrac{\phi'\delta^{i_1}_{j_1}}{D} - \dfrac{\rho u^{i_1}_{j_1}}{D\phi} + \dfrac{\phi'\rho^2 u^{i_1}u_{j_1}}{D^3} + \dfrac{\rho^3 u^{i_1}u_p u^{p}_{j_1}}{D^3\phi}\right)\notag \\
		&&\cdots \left( \dfrac{\phi'\delta^{i_k}_{j_k}}{D} - \dfrac{\rho u^{i_k}_{j_k}}{D\phi} + \dfrac{\phi'\rho^2 u^{i_k}u_{j_k}}{D^3} + \dfrac{\rho^3 u^{i_k}u_p u^{p}_{j_k}}{D^3\phi}\right). \label{sk12}
	\end{eqnarray}
	
	By (\ref{Sigmak}), we know the term\ $\dfrac{\phi'\rho^2 u^iu_j}{D^3}$ or $\dfrac{\rho^3 u^iu_p u^p_j}{D^3\phi}$ occurs at most once in each sum. Thus we calculate (\ref{sk12}) in three cases.
	
	(1) $m\geqslant 0$ times $-\dfrac{\rho u^i_j}{D\phi}$, others are $\dfrac{\phi' \delta^i_j}{D}$. Note that
\begin{equation}
	\delta_{i_1i_2\cdots i_mi_{m+1}\cdots i_k}^{j_1j_2\cdots j_mj_{m+1}\cdots j_k}\delta^{i_{m+1}}_{j_{m+1}}\cdots\delta^{i_k}_{j_k} = \delta_{i_1i_2\cdots i_m}^{j_1j_2\cdots j_m}{n-m \choose k-m}(k-m)!,
\end{equation}
 the sum is
 \begin{equation}
 	\begin{aligned}
 			A_m=&{k\choose m}\dfrac{1}{k!}\delta_{i_1i_2\cdots i_k}^{j_1j_2\cdots j_k} \dfrac{-\rho u^{i_1}_{j_1}}{D\phi}\dfrac{-\rho u^{i_2}_{j_2}}{D\phi}\cdots \dfrac{-\rho u^{i_m}_{j_m}}{D\phi}\dfrac{\phi'\delta^{i_{m+1}}_{j_{m+1}}}{D}\cdots \dfrac{\phi'\delta^{i_{m+1}}_{j_{m+1}}}{D}\\
 			=&{k\choose m}\dfrac{(-1)^m\rho^m \phi'^{k-m}}{D^k\phi^m}\dfrac{1}{k!} \delta_{i_1i_2\cdots i_k}^{j_1j_2\cdots j_k} u^{i_1}_{j_1}\cdots u^{i_m}_{j_m}\delta^{i_{m+1}}_{j_{m+1}}\cdots\delta^{i_k}_{j_k}\\
 			=&{k\choose m}\dfrac{(-1)^m\rho^m \phi'^{k-m}}{D^k\phi^m}\dfrac{1}{k!} \delta_{i_1i_2\cdots i_m}^{j_1j_2\cdots j_m} u^{i_1}_{j_1}\cdots u^{i_m}_{j_m}{n-m \choose k-m}(k-m)!\\
 			=&\dfrac{(-1)^m\rho^m \phi'^{k-m}}{D^k\phi^m}{n-m\choose k-m}\sigma_m(D^2 u).
 	\end{aligned}
 \end{equation}
	
	(2) Once $\dfrac{\phi'\rho^2 u^iu_j}{D^3}$, $m\geqslant 0$ times $-\dfrac{\rho u^i_j}{D\phi}$, others are\ $\dfrac{\phi'\delta^i_j}{D}$. Note that
	\begin{equation}
		\delta_{i_1i_2\cdots i_{m+1}i_{m+2}\cdots i_k}^{j_1j_2\cdots j_{m+1}j_{m+2}\cdots j_k}\delta^{i_{m+2}}_{j_{m+2}}\cdots\delta^{i_k}_{j_k} = \delta_{i_1i_2\cdots i_{m+1}}^{i_1j_2\cdots j_{m+1}}{n-m-1\choose k-m-1} (k-m-1)!,  \label{nmidentity}
	\end{equation}
 the sum is
 \begin{equation}
 	\begin{aligned}
 		B_m=&k{k-1\choose m}\dfrac{1}{k!}\delta_{i_1i_2\cdots i_k}^{j_1j_2\cdots j_k}\dfrac{\phi'\rho^2 u^{i_1}u_{j_1}}{D^3}\dfrac{-\rho u^{i_2}_{j_2}}{D\phi}\cdots \dfrac{-\rho u^{i_{m+1}}_{j_{m+1}}}{D\phi}\dfrac{\phi'\delta^{i_{m+2}}_{j_{m+2}}}{D}\cdots \dfrac{\phi'\delta^{i_k}_{j_k}}{D}\\
 		=&{k-1\choose m}\dfrac{(-1)^m\rho^{m+2}\phi'^{k-m}}{D^{k+2}\phi^m}\dfrac{1}{(k-1)!}\delta_{i_1i_2\cdots i_k}^{j_1j_2\cdots j_k} u^{i_1}u_{j_1}u^{i_2}_{j_2}\cdots u^{i_{m+1}}_{j_{m+1}}\delta^{i_{m+2}}_{j_{m+2}}\cdots\delta^{i_k}_{j_k}\\
 		=&{k-1\choose m}\dfrac{(-1)^m\rho^{m+2}\phi'^{k-m}}{D^{k+2}\phi^m}\dfrac{1}{(k-1)!}\delta_{i_1i_2\cdots i_{m+1}}^{j_1j_2\cdots j_{m+1}} u^{i_1}u_{j_1}u^{i_2}_{j_2}\cdots u^{i_{m+1}}_{j_{m+1}}\\
 		&\cdot {n-m-1\choose k-m-1} (k-m-1)!\\
 		=&\dfrac{(-1)^m\rho^{m+2}\phi'^{k-m}}{D^{k+2}\phi^m}{n-m-1\choose k-m-1}u^iu_j[T_m]^j_i(D^2 u).
 	\end{aligned}
 \end{equation}
	
	(3) Once $\dfrac{\rho^3 u^iu_k u^k_j}{D^3\phi}$, $m\geqslant 0$ times $-\dfrac{\rho u^i_j}{D\phi}$, others are\ $\dfrac{\phi'\delta^i_j}{D}$. By using (\ref{nmidentity}) again, the sum is
	\begin{equation}
		\begin{aligned}
			C_m=&k{k-1\choose m}\dfrac{1}{k!}\delta_{i_1i_2\cdots i_k}^{j_1j_2\cdots j_k}\dfrac{\rho^3 u^{i_1}u_p u^p_{j_1}}{D^3\phi}\dfrac{-\rho u^{i_2}_{j_2}}{D\phi}\cdots \dfrac{-\rho u^{i_{m+1}}_{j_{m+1}}}{D\phi}\dfrac{\phi'\delta^{i_{m+2}}_{j_{m+2}}}{D}\cdots \dfrac{\phi'\delta^{i_k}_{j_k}}{D}\\
			=&{k-1\choose m}\dfrac{(-1)^m\rho^{m+3}\phi'^{k-m-1}}{D^{k+2}\phi^{m+1}}\dfrac{1}{(k-1)!}\delta_{i_1i_2\cdots i_k}^{j_1j_2\cdots j_k}u^{i_1}u_p u^p_{j_1}u^{i_2}_{j_2}\cdots u^{i_{m+1}}_{j_{m+1}}\delta^{i_{m+2}}_{j_{m+2}}\cdots\delta^{i_k}_{j_k}\\
			=&{k-1\choose m}\dfrac{(-1)^m\rho^{m+3}(\phi')^{k-m-1}}{D^{k+2}\phi^{m+1}}\dfrac{1}{(k-1)!}\delta_{i_1i_2\cdots i_{m+1}}^{j_1j_2\cdots j_{m+1}}u^{i_1}u_p u^p_{j_1}u^{i_2}_{j_2}\cdots u^{i_{m+1}}_{j_{m+1}}\\
			&\cdot {n-m-1\choose k-m-1}(k-m-1)!\\
			=&\dfrac{(-1)^m\rho^{m+3}\phi'^{k-m-1}}{D^{k+2}\phi^{m+1}}{n-m-1\choose k-m-1}u^iu_pu^p_j[T_m]_i^j(D^2 u).
		\end{aligned}
	\end{equation}
	
	(4) Other cases are all 0.
	
	Hence
	\begin{eqnarray*}
		\sigma_k(h^i_j) &=& \sum\limits_{m=0}^k A_m+B_m+C_m\\
		&=& \sum\limits_{m=0}^k \dfrac{(-1)^m\rho^m\phi'^{k-m}}{D^{k+2}\phi^m}{n-m\choose k-m}\left[\phi^2\sigma_m(D^2 u) + \rho^2|\nabla u|^2 \sigma_m(D^2 u) \right]\\
		&&+\sum\limits_{m=0}^k \dfrac{(-1)^m\rho^{m+2}\phi'^{k-m}}{D^{k+2}\phi^m}{n-m\choose k-m}\dfrac{k-m}{n-m} u^iu_j[T_m]^j_i(D^2 u)\\
		&&+ \sum\limits_{m=1}^k \dfrac{(-1)^{m-1}\rho^{m+2}\phi'^{k-m}}{D^{k+2}\phi^{m}}{n-m\choose k-m}u^iu_pu^p_j[T_{m-1}]_i^j(D^2 u).
	\end{eqnarray*}
	Using \eqref{newton}, we get the conclusion (\ref{sigmak}).
\end{proof}

\subsection{Computation of curvature integrals and weighted curvature integrals}
%Thanks to $\|u\|_{W^{2,\infty}(\mathbb{S}^n)}<\varepsilon$, we can use Taylor expansion to get the expression for curvature integrals$\displaystyle\int_M \sigma_k(\kappa)\mathrm{d}\mu_g$ and weighted curvature integrals $\displaystyle\int_M \Phi\sigma_k(\kappa) \mathrm{d}\mu_g,\ \int_M \phi'\sigma_k(\kappa) \mathrm{d}\mu_g \ (0\leqslant k\leqslant n)$.

\begin{lem}[Expression for $\displaystyle\int_{M}\sigma_k(\kappa) \mathrm{d}\mu_g\ (0\leqslant k\leqslant n)$]
	Let\ $\Omega\subset\mathbb{N}^{n+1}(K)\ (K=-1,1)$, $M=\partial\Omega=\{(\rho(1+u(x)),x):x\in\mathbb{S}^n\}$, where $u\in C^3(\mathbb{S}^n)$. Suppose $\|u\|_{W^{2,\infty}(\mathbb{S}^n)}<\varepsilon$, then for any\ $0\leqslant k\leqslant n$, there holds
	\begin{eqnarray}
		&&\displaystyle\int_M \sigma_k(\kappa) \mathrm{d}\mu_g\notag \\
		&=& \int_{\mathbb{S}^n} {n\choose k}\phi^{n-k}(\rho)\phi'^ k(\rho) \mathrm{d}A \notag\\
		&&+ \int_{\mathbb{S}^n} {n\choose k}\left[ (n-k)\phi^{n-k-1}(\rho)\phi'^{k+1}(\rho)-Kk\phi^{n-k+1}(\rho)\phi'^{k-1}(\rho) \right] \rho u \mathrm{d}A  \notag\\
		&&+ \int_{\mathbb{S}^n} {n\choose k} \left[\dfrac{(n-k)(n-k-1)}{2}\phi^{n-k-2}(\rho)\phi'^{k+2}(\rho) \right.\notag\\
		&&\ \ \ \ \ \ \ \ \ \ \ \ \ \ \ + K(k^2-kn-\dfrac{n}{2})\phi^{n-k}(\rho)\phi'^k(\rho) \notag\\
		&&\ \ \ \ \ \ \ \ \ \ \ \ \ \ \ + \left.\dfrac{k(k-1)}{2}\phi^{n-k+2}(\rho)\phi'^{k-2}(\rho) \right]\rho^2 u^2\mathrm{d}A  \notag\\
		&&+ \int_{\mathbb{S}^n} {n\choose k} \left[ \dfrac{(n-k)(k+1)}{2n}\phi^{n-k-2}(\rho)\phi'^k(\rho) \right.\left.- K \dfrac{k(k-1)}{2n}\phi^{n-k}(\rho)\phi'^{k-2}(\rho) \right]\rho^2 |\nabla u|^2 \mathrm{d}A \notag\\
		%&&\ \ \ \ \left.-K \dfrac{(k-m)(k-m-1)}{2(n-m)(m+1)}\phi^{n-m-k}(\rho)\phi'^{k-m-2}(\rho) \right]\rho^{m+2}|\nabla u|^2\sigma_m(D^2 u) \mathrm{d}A \notag \\
		&&+ O(\varepsilon)\|u\|_{L^2(\mathbb{S}^n)}^2 + O(\varepsilon)\|\nabla u\|_{L^2(\mathbb{S}^n)}^2.   \label{ak111}
	\end{eqnarray}
	
	%\end{comment}
\end{lem}

\begin{proof}
	By \eqref{area-elem} and \eqref{sigmak}, we have
	\begin{eqnarray*}
		\sigma_k(\kappa) \mathrm{d}\mu_g &=& \sum\limits_{m=0}^k (-1)^m {n-m \choose k-m}\frac{\phi'^{k-m}}{D^{k+1}}  \notag\\
		&&\times \left[\rho^m\phi^{n-m+1}\sigma_m(D^2u) + \dfrac{k+n-2m}{n-m} \rho^{m+2} \phi^{n-m-1}u^iu_j[T_m]_i^j(D^2 u)\right]\mathrm{d}A,
	\end{eqnarray*}
	where $r=\rho(1+u)$, $\phi=\phi(r)$, $\phi'=\phi'(r)$, $D=\sqrt{\phi^2(r)+\rho^2|\nabla u|^2}$ can all be seen as functions of two independent variables $u$ and $|\nabla u|^2$. We can use Taylor expansion to expand the equation above at $u=|\nabla u|^2=0$ (i.e. at $r\equiv\rho$), and use $\|u\|_{W^{2,\infty}(\mathbb{S}^n)}<\varepsilon$ to control the remainder terms, then get
	\begin{equation}\label{int-sigmak1}
		\begin{aligned}
			\displaystyle\int_M \sigma_k(\kappa) \mathrm{d}\mu_g=& \int_{\mathbb{S}^n}\sum\limits_{m=0}^k \left[\left(A_0^m+A_1^mu+A_2^mu^2+A^m|\nabla u|^2\right)\sigma_m(D^2u)+B^mu^iu_j[T_m]_i^j(D^2 u)\right]\mathrm{d}A\\
			&+O(\varepsilon)\|u\|_{L^2(\mathbb{S}^n)}^2+O(\varepsilon)\|\nabla u\|_{L^2(\mathbb{S}^n)}^2,
		\end{aligned}
	\end{equation}
	where for $i=0,1,2$,
	\begin{equation}\label{A_i^m}
		\begin{aligned}
			A_i^m=&\left(i!\right)^{-1}\left.\frac{\partial^i}{\partial^i u} \left[(-1)^m\rho^m {n-m \choose k-m}\frac{\phi'^{k-m}}{D^{k+1}}\phi^{n-m+1}\right]\right|_{u=|\nabla u|^2=0}\\
			=&\left(i!\right)^{-1}(-1)^m\rho^{m} {n-m \choose k-m}\left.\frac{\mathrm{d}^i}{\mathrm{d}^i r}\left[\phi'^{k-m}\phi^{n-m+1}\right]\right|_{r=\rho}\rho^i,
		\end{aligned}
	\end{equation}
	\begin{equation}\label{A^m}
		\begin{aligned}
			A^m=&\frac{\partial}{\partial \left(|\nabla u|^2\right)}\left. \left[(-1)^m\rho^m {n-m \choose k-m}\frac{\phi'^{k-m}}{D^{k+1}}\phi^{n-m+1}\right]\right|_{u=|\nabla u|^2=0}\\
			=&(-1)^{m+1}\rho^{m+2}\frac{k+1}{2}{n-m \choose k-m}\phi'^{k-m}(\rho)\phi^{n-m-k-2}(\rho),
		\end{aligned}
	\end{equation}
	\begin{equation}\label{B^m}
		\begin{aligned}
			B^m=&\left.(-1)^m\rho^{m+2}{n-m \choose k-m}\dfrac{k+n-2m}{n-m} \left[\frac{\phi'^{k-m}}{D^{k+1}}\phi^{n-m-1}\right]\right|_{u=|\nabla u|^2=0}\\
			=&(-1)^m\rho^{m+2}{n-m \choose k-m}\dfrac{k+n-2m}{n-m} \phi'^{k-m}(\rho)\phi^{n-m-1}(\rho).
		\end{aligned}
	\end{equation}
	
	Using the following identities which have been proved in \cite{VanBlargan2022QuantitativeQI} Lemma 4.2:
	\begin{eqnarray}
		\int_{\mathbb{S}^n} u^iu_j[T_m]_i^j(D^2 u)\mathrm{d}A &=& \dfrac{m+2}{2} \int_{\mathbb{S}^n} |\nabla u|^2\sigma_m(D^2 u) \mathrm{d}A \notag\\
		&& + O(\varepsilon)\|\nabla u\|_{L^2(\mathbb{S}^n)}^2\ \ (m\geqslant 1),\\
		\int_{\mathbb{S}^n} \sigma_m(D^2 u) \mathrm{d}A &=& \dfrac{n-m+1}{2}\int_{\mathbb{S}^n} |\nabla u|^2\sigma_{m-2}(D^2 u)\mathrm{d}A \notag \\
		&& +O(\varepsilon) \|\nabla u\|_{L^2(\mathbb{S}^n)}^2\ \ (m\geqslant 2),\\
		\int_{\mathbb{S}^n} \sigma_1(D^2 u) \mathrm{d}A &=& 0,\\
		\int_{\mathbb{S}^n} u\sigma_m(D^2 u) \mathrm{d}A &=& -\dfrac{m+1}{2m}\int_{\mathbb{S}^n} |\nabla u|^2\sigma_{m-1}(D^2 u) \mathrm{d}A \notag\\
		&&+ O(\varepsilon) \|\nabla u\|_{L^2(\mathbb{S}^n)}^2\ \ (m\geqslant 1),\\
		\int_{\mathbb{S}^n} u^2\sigma_m(D^2 u) \mathrm{d}A &=& O(\varepsilon) \|\nabla u\|_{L^2(\mathbb{S}^n)}^2\ \ (m\geqslant 1),
	\end{eqnarray}
	we can rewrite \eqref{int-sigmak1} as
	\begin{eqnarray}
		&&\displaystyle\int_M \sigma_k(\kappa) \mathrm{d}\mu_g \notag \\
		&=&\int_{\mathbb{S}^n} \left[A_0^0+A_1^0u+A_2^0u^2+(A^0+B^0)|\nabla u|^2\right]\mathrm{d}A  \notag\\
		&&+\int_{\mathbb{S}^n} \left[-A_1^1|\nabla u|^2+(A^1+\frac{3}{2}B^1)|\nabla u|^2\sigma_1(D^2u)\right]\mathrm{d}A \notag\\
		&&+\int_{\mathbb{S}^n}\sum\limits_{m=2}^k\left[A_0^m\frac{n-m+1}{2}|\nabla u|^2\sigma_{m-2}(D^2u)-A_1^m\frac{m+1}{2m}|\nabla u|^2\sigma_{m-1}(D^2u)\right. \notag\\
		&&\quad\quad\quad\quad\quad\left.+A^m|\nabla u|^2\sigma_{m}(D^2u)+B^m\frac{m+2}{2}|\nabla u|^2\sigma_{m}(D^2u)\right]\mathrm{d}A \notag\\
		&&+O(\varepsilon)\|u\|_{L^2(\mathbb{S}^n)}^2+O(\varepsilon)\|\nabla u\|_{L^2(\mathbb{S}^n)}^2\notag\\
		&=&\int_{\mathbb{S}^n} \left[A_0^0+A_1^0u+A_2^0u^2\right]\mathrm{d}A \notag\\
		&&+\sum\limits_{m=0}^{k-2}\int_{\mathbb{S}^n}\left[A^m+B^m\frac{m+2}{2}-A_1^{m+1}\frac{m+2}{2(m+1)}+A_0^{m+2}\frac{n-m-1}{2}\right]|\nabla u|^2\sigma_m(D^2 u)\mathrm{d}A \notag\\
		&&+\int_{\mathbb{S}^n}\left[A^{k-1}+B^{k-1}\frac{k+1}{2}-A_1^{k}\frac{k+1}{2k}\right]|\nabla u|^2\sigma_{k-1}(D^2 u)\mathrm{d}A\notag\\
		&&+\int_{\mathbb{S}^n}\left[A^{k}+B^{k}\frac{k+2}{2}\right]|\nabla u|^2\sigma_k(D^2 u) \mathrm{d}A \notag\\
		&&+O(\varepsilon)\|u\|_{L^2(\mathbb{S}^n)}^2+O(\varepsilon)\|\nabla u\|_{L^2(\mathbb{S}^n)}^2.
	\end{eqnarray}
	
	\begin{comment}
		\begin{equation}
			\begin{aligned}
				\displaystyle\int_M \sigma_k(\kappa) \mathrm{d}\mu_g=&\int_{\mathbb{S}^n} \left[A_0^0+A_1^0u+A_2^0u^2+(A^0+B^0)|\nabla u|^2\right]\mathrm{d}A\\
				&+\int_{\mathbb{S}^n} \left[-A_1^1|\nabla u|^2+(A^1+\frac{3}{2}B^1)|\nabla u|^2\sigma_1(D^2u)\right]\mathrm{d}A\\
				&+\int_{\mathbb{S}^n}\sum\limits_{m=2}^k\left[A_0^m\frac{n-m+1}{2}|\nabla u|^2\sigma_{m-2}(D^2u)-A_1^m\frac{m+1}{2m}|\nabla u|^2\sigma_{m-1}(D^2u)\right.\\
				&\quad\quad\quad\quad\quad\left.+A^m|\nabla u|^2\sigma_{m}(D^2u)+B^m\frac{m+2}{2}|\nabla u|^2\sigma_{m}(D^2u)\right]\mathrm{d}A\\
				&+O(\varepsilon)\|u\|_{L^2(\mathbb{S}^n)}^2+O(\varepsilon)\|\nabla u\|_{L^2(\mathbb{S}^n)}^2\\
				=&\int_{\mathbb{S}^n} \left[A_0^0+A_1^0u+A_2^0u^2\right]\mathrm{d}A\\
				&+\sum\limits_{m=0}^{k-2}\int_{\mathbb{S}^n}\left[A^m+B^m\frac{m+2}{2}-A_1^{m+1}\frac{m+2}{2(m+1)}+A_0^{m+2}\frac{n-m-1}{2}\right]\mathrm{d}A\\
				&+\int_{\mathbb{S}^n}\left[A^{k-1}+B^{k-1}\frac{k+1}{2}-A_1^{k}\frac{k+1}{2k}\right]\mathrm{d}A+\int_{\mathbb{S}^n}\left[A^{k}+B^{k}\frac{k+2}{2}\right]\mathrm{d}A\\
				&+O(\varepsilon)\|u\|_{L^2(\mathbb{S}^n)}^2+O(\varepsilon)\|\nabla u\|_{L^2(\mathbb{S}^n)}^2.
			\end{aligned}
		\end{equation}
	\end{comment}
	
	Note that for $1\leqslant k\leqslant n$,
	$$ \sum\limits_{m=1}^k \int_{\mathbb{S}^n}  C(n,m,k,\rho)|\nabla u|^2 \sigma_m(D^2 u) \mathrm{d}A =O(\varepsilon)\|\nabla u\|_{L^2(\mathbb{S}^n)}^2,$$
	we have
	\begin{eqnarray}
		\displaystyle\int_M \sigma_k(\kappa) \mathrm{d}\mu_g&=&
		\int_{\mathbb{S}^n} \left[A_0^0+A_1^0u+A_2^0u^2\right]\mathrm{d}A+\int_{\mathbb{S}^n}\left[A^0+B^0-A_1^{1}+A_0^{2}\frac{n-1}{2}\right]|\nabla u|^2\mathrm{d}A \notag\\
		&&+O(\varepsilon)\|u\|_{L^2(\mathbb{S}^n)}^2+O(\varepsilon)\|\nabla u\|_{L^2(\mathbb{S}^n)}^2.
	\end{eqnarray}
	Then by \eqref{A_i^m}, \eqref{A^m}, \eqref{B^m} and direct calculation, using \eqref{hd}, we get the formula \eqref{ak111}.
\end{proof}

%Here we use Taylor expansion to derive the expression for $\displaystyle\int_M  \Phi(\rho(1+u))\sigma_k(\kappa) \mathrm{d}\mu_g\ (1\leqslant k \leqslant n)$.

\begin{lem}[Expression for $\displaystyle\int_M \Phi\sigma_k(\kappa)\mathrm{d}\mu_g\ (0\leqslant k\leqslant n)$]
	Let $M=\{(\rho(1+u(x)),x):x\in\mathbb{S}^n\}$ be a nearly spherical set in $\mathbb{N}^{n+1}(K)\ (K=1,0,-1)$, where $u\in C^3(\mathbb{S}^n)$. $\Phi(r)$ is defined in \eqref{Phi}. Suppose $\|u\|_{W^{2,\infty}(\mathbb{S}^n)}<\varepsilon$, then for any $0\leqslant k\leqslant n$, there holds
	\begin{comment}
		\begin{eqnarray}
			\int_M \Phi(\rho(1+u))\sigma_k(\kappa) \mathrm{d}\mu_g &=& \int_{\mathbb{S}^n} {n\choose k} \phi^{n-k}(\rho) \phi'^k(\rho) \Phi(\rho) \mathrm{d}A \notag\\
			&& + \int_{\mathbb{S}^n} {n\choose k} \left\lbrace  \left[(n-k)\phi^{n-k-1}(\rho)\phi'^{k+1}(\rho) - Kk\phi^{n-k+1}(\rho)\phi'^{k-1}(\rho) \right]\Phi(\rho) \right.  \notag \\
			&&\ \ \ \ \left. + \phi^{n-k+1}(\rho)\phi'^k(\rho) \right\rbrace \rho u\mathrm{d}A  \notag\\
			&&+ \int_{\mathbb{S}^n} {n\choose k} \left\lbrace \left[ \dfrac{(n-k)(n-k-1)}{2}\phi^{n-k-2}(\rho)\phi'^{k+2}(\rho) \right.\right.\notag\\
			&&\ \ \ \ \left.\left. +K(k^2-kn-\dfrac{n}{2})\phi^{n-k}(\rho)\phi'^k(\rho) + \dfrac{k(k-1)}{2}\phi^{n-k+2}(\rho)\phi'^{k-2}(\rho) \right]\Phi(\rho) \right. \notag\\
			&&\ \ \ \ \left.+(n-k+\frac{1}{2})\phi^{n-k}(\rho)\phi'^{k+1}(\rho) - Kk\phi^{n-k+2}(\rho)\phi'^{k-1}(\rho) \right\rbrace \rho^2 u^2 \mathrm{d}A \notag\\
			&&+\sum\limits_{m=0}^k (-1)^m \rho^{m+2} {n-m \choose k-m} \int_{\mathbb{S}^n}  \left\lbrace \left[ \dfrac{(n-k)(k+1)}{2(n-m)(m+1)}\phi^{n-m-k-2}(\rho)\phi'^{k-m}(\rho) \right.\right.\notag\\
			&&\ \ \ \ \left.-\dfrac{K(k-m)(k-m-1)}{2(n-m)(m+1)}\phi^{n-m-k}(\rho)\phi'^{k-m-2}(\rho)  \right] \Phi(\rho) \notag\\
			&&\ \ \ \ \left.+\dfrac{(k-m)(m+2)}{2(n-m)(m+1)} \phi^{n-m-k}(\rho)\phi'^{k-m-1}(\rho) \right\rbrace |\nabla u|^2 \sigma_m(D^2 u) \mathrm{d}A \notag\\
			&&+O(\varepsilon)\|u\|_{L^2(\mathbb{S}^n)}^2 + O(\varepsilon) \|\nabla u\|_{L^2(\mathbb{S}^n)}^2. \label{PhiN}
		\end{eqnarray}
	\end{comment}
	
	\begin{eqnarray}
		%&&\int_M \Phi\sigma_k(\kappa) \mathrm{d}\mu_g - \int_{\partial\overline{B}_{\rho}} \Phi\sigma_k(\kappa) \mathrm{d}\mu_g \notag \\
		&&\int_M \Phi(r)\sigma_k(\kappa) \mathrm{d}\mu_g \notag\\
		&=& \int_{\mathbb{S}^n} {n\choose k} \phi^{n-k}(\rho) \phi'^k(\rho) \Phi(\rho) \mathrm{d}A \notag\\
		&&+ \int_{\mathbb{S}^n} {n\choose k} \left\lbrace  \left[(n-k)\phi^{n-k-1}(\rho)\phi'^{k+1}(\rho) - Kk\phi^{n-k+1}(\rho)\phi'^{k-1}(\rho) \right]\Phi(\rho) \right.\notag\\
		&&\ \ \ \ \ \ \ \ \ \ \ \ \ \ \ \left.+ \phi^{n-k+1}(\rho)\phi'^k(\rho) \right\rbrace \rho u\mathrm{d}A  \notag\\
		&&+ \int_{\mathbb{S}^n} {n\choose k} \left\lbrace \left[ \dfrac{(n-k)(n-k-1)}{2}\phi^{n-k-2}(\rho)\phi'^{k+2}(\rho) \right.\right.\notag\\
		&&\ \ \ \ \ \ \ \ \ \ \ \ \ \ \ \left.\left. +K(k^2-kn-\dfrac{n}{2})\phi^{n-k}(\rho)\phi'^k(\rho) + \dfrac{k(k-1)}{2}\phi^{n-k+2}(\rho)\phi'^{k-2}(\rho) \right]\Phi(\rho) \right. \notag\\
		&&\ \ \ \ \ \ \ \ \ \ \ \ \ \ \ \left.+(n-k+\frac{1}{2})\phi^{n-k}(\rho)\phi'^{k+1}(\rho) - Kk\phi^{n-k+2}(\rho)\phi'^{k-1}(\rho) \right\rbrace \rho^2 u^2 \mathrm{d}A \notag\\
		&&+\int_{\mathbb{S}^n} {n \choose k} \left\lbrace \left[ \dfrac{(n-k)(k+1)}{2n}\phi^{n-k-2}(\rho)\phi'^{k}(\rho) \right. -K\dfrac{k(k-1)}{2n}\phi^{n-k}(\rho)\phi'^{k-2}(\rho)  \right] \Phi(\rho) \notag\\
		&&\ \ \ \ \ \ \ \ \ \ \ \ \ \ \ \left.+\dfrac{k}{n} \phi^{n-k}(\rho)\phi'^{k-1}(\rho) \right\rbrace \rho^{2}|\nabla u|^2 \mathrm{d}A \notag\\
		&&+ O(\varepsilon)\|u\|_{L^2(\mathbb{S}^n)}^2 + O(\varepsilon) \|\nabla u\|_{L^2(\mathbb{S}^n)}^2. \label{PhiinN}
	\end{eqnarray}
	
\end{lem}
\begin{proof}
	Notice that
	\begin{eqnarray}
		&&\int_M \Phi\sigma_k(\kappa) \mathrm{d}\mu_g \notag\\
		&=& \int_{\mathbb{S}^n} \sum\limits_{m=0}^k \dfrac{ (-1)^m\rho^m \phi'^{k-m}(\rho(1+u))\phi^{n-m-1}(\rho(1+u)) \Phi(\rho(1+u)) }{\left(\phi^2+\rho^2|\nabla u|^2\right)^{\frac{k+1}{2}}}{n-m \choose k-m} \notag\\
		&&\times \left[ \phi^2(\rho(1+u))\sigma_m(D^2 u) + \dfrac{k+n-2m}{n-m} \rho^2 u^iu_j[T_m]_i^j(D^2 u) \right] \mathrm{d}A,
	\end{eqnarray}
	we expand $\Phi(\rho(1+u))$ at $u=0$ as
	\begin{equation}
		\Phi(\rho(1+u)) = \Phi(\rho) + \phi(\rho) \rho u + \dfrac{1}{2} \phi'(\rho) \rho^2 u^2 + o(u^2),
	\end{equation}
	and after a similar computation as for (\ref{ak111}), we get the conclusion.
\end{proof}

\begin{lem}[Expression for $\displaystyle\int_M \phi'\sigma_k(\kappa)\mathrm{d}\mu_g\ (0\leqslant k\leqslant n)$]
	Let $M=\{(\rho(1+u(x)),x):x\in\mathbb{S}^n\}$ be a nearly spherical set in $\mathbb{N}^{n+1}(K)\ (K=1,-1)$, where $u\in C^3(\mathbb{S}^n)$. Suppose $\|u\|_{W^{2,\infty}(\mathbb{S}^n)}<\varepsilon$, then for any $0\leqslant k\leqslant n$, there holds
	\begin{eqnarray}
		&&\int_M \phi'(r)\sigma_k(\kappa) \mathrm{d}\mu_g \notag\\
		&=& \int_{\mathbb{S}^n} {n\choose k} \phi^{n-k}(\rho)\phi'^{k+1}(\rho) \mathrm{d}A \notag\\
		&& + \int_{\mathbb{S}^n} {n\choose k} \left[ (n-k)\phi^{n-k-1}(\rho)\phi'^{k+2}(\rho) \right.
		\left.- K(k+1)\phi^{n-k+1}(\rho)\phi'^k(\rho) \right] \rho u \mathrm{d}A \notag\\
		&&+ \int_{\mathbb{S}^n} {n\choose k}\left[ \dfrac{(n-k)(n-k-1)}{2}\phi^{n-k-2}(\rho)\phi'^{k+3}(\rho)  \right. \notag\\
		&&\ \ \ \ \ \ \ \ \ \ \ \ \ \ +\left. K(k^2-kn+k-\dfrac{3}{2}n-\dfrac{1}{2})\phi^{n-k}(\rho)\phi'^{k+1}(\rho) \right.\notag\\
		&&\ \ \ \ \ \ \ \ \ \ \ \ \ \ \left.+ \dfrac{k(k+1)}{2}\phi^{n-k+2}(\rho)\phi'^{k-1}(\rho) \right] \rho^2u^2 \mathrm{d}A  \notag\\
		&&+\int_{\mathbb{S}^n} {n\choose k} \left[ \dfrac{(n-k)(k+1)}{2n} \phi^{n-k-2}(\rho)\phi'^{k+1}(\rho) \right. \left.- K\dfrac{k(k+1)}{2n}\phi^{n-k}(\rho)\phi'^{k-1}(\rho) \right]\rho^2|\nabla u|^2 \mathrm{d}A \notag \\
		&&+ O(\varepsilon) \|u\|_{L^2(\mathbb{S}^n)}^2 + O(\varepsilon) \|\nabla u\|_{L^2(\mathbb{S}^n)}^2.  \label{phi'sigmak}
	\end{eqnarray}
\end{lem}
\begin{proof}
	Notice that
	\begin{eqnarray}
		&&\int_M \phi'(r)\sigma_k(\kappa) \mathrm{d}\mu_g \notag\\
		&=& \int_{\mathbb{S}^n} \sum\limits_{m=0}^k \dfrac{ (-1)^m\rho^m \phi'^{k-m+1}(\rho(1+u))\phi^{n-m-1}(\rho(1+u))}{\left(\phi^2+\rho^2|\nabla u|^2\right)^{\frac{k+1}{2}}}{n-m \choose k-m}  \notag\\
		&&\times \left[ \phi^2(\rho(1+u))\sigma_m(D^2 u) + \dfrac{k+n-2m}{n-m} \rho^2 u^iu_j[T_m]_i^j(D^2 u) \right] \mathrm{d}A,
	\end{eqnarray}
	we expand $\phi'(\rho(1+u))$ at $u=0$ as
	\begin{equation}
		\phi'(\rho(1+u)) = \phi'(\rho) - K\phi(\rho) \rho u - \dfrac{1}{2} K\phi'(\rho) \rho^2 u^2 + o(u^2),
	\end{equation}
	and after a similar computation as for (\ref{ak111}), we get the conclusion.
\end{proof}

%\begin{remark}
%	(\ref{a0111}) is the special case of (\ref{ak111}) when $k=0$.
%\end{remark}

\subsection{Computation of $\mathscr{A}_k(\Omega)-\mathscr{A}_k(\overline{B}_{\rho})$}
First we compute the expression for $\mathscr{A}_k(\Omega)-\mathscr{A}_k(\overline{B}_{\rho})\ (-1\leqslant k\leqslant 1)$, then use iteration process to derive the expression for $k\leqslant n$. Note that the principal curvatures of\ $\partial\overline{B}_{\rho}$ are\ $\kappa_1=\kappa_2=\cdots =\kappa_n = \dfrac{\phi'(\rho)}{\phi(\rho)}$.

From (\ref{a-1omega}), we get
\begin{eqnarray}
	\mathscr{A}_{-1}(\Omega)-\mathscr{A}_{-1}(\overline{B}_{\rho})
	&=& \int_{\mathbb{S}^n}\left(\int_{\rho}^{\rho(1+u)} \phi^n(r)\mathrm{d}r \right)\mathrm{d}A\notag\\
	&=& \int_{\mathbb{S}^n} \phi^n(\rho)\rho u \mathrm{d}A+\int_{\mathbb{S}^n} \dfrac{n}{2}\phi^{n-1}(\rho)\phi'(\rho)\rho^2 u^2 \mathrm{d}A \notag \\
	&&+ O(\varepsilon)\|u\|_{L^2(\mathbb{S}^n)}^2.  \label{a-1ob}
\end{eqnarray}

Since $\mathscr{A}_0(\Omega)=\displaystyle\int_{M}\sigma_0(\kappa)\mathrm{d}\mu_g$, we can take $k=0$ in \eqref{ak111} and get
\begin{eqnarray}
	\mathscr{A}_{0}(\Omega)-\mathscr{A}_{0}(\overline{B}_{\rho})
	&=&  \int_{\mathbb{S}^n} n\phi^{n-1}(\rho)\phi'(\rho) \rho u\mathrm{d}A \notag \\
	&&+ \int_{\mathbb{S}^n} \left[\dfrac{1}{2} \left(n(n-1)\phi^{n-2}(\rho)\phi'^2(\rho) -nK\phi^n(\rho) \right)\right] \rho^2 u^2\mathrm{d}A \notag\\
	&&+ \int_{\mathbb{S}^n} \dfrac{1}{2} \phi^{n-2}(\rho) \rho^2|\nabla u|^2 \mathrm{d}A \notag\\
	&&+ O(\varepsilon)\|u\|_{L^2(\mathbb{S}^n)}^2 + O(\varepsilon)\|\nabla u\|_{L^2(\mathbb{S}^n)}^2.  \label{a0ob}
\end{eqnarray}

Combining (\ref{a1omega}), taking $k=1$ in (\ref{ak111}) and inserting \eqref{a-1ob}, we get
\begin{eqnarray}
	&&\mathscr{A}_{1}(\Omega)-\mathscr{A}_{1}(\overline{B}_{\rho}) \notag\\ %&=&\int_{M}\sigma_1(\kappa)\mathrm{d}\mu_g+Kn(\mathscr{A}_{-1}(\Omega)-\mathscr{A}_{-1}(\overline{B}_{\rho}))\notag\\
	&=& \int_{\mathbb{S}^n} {n\choose 1}\left[(n-1)\phi^{n-2}(\rho)\phi'^2(\rho) -K\phi^n(\rho)  \right] \rho u \mathrm{d}A \notag\\
	&& + \int_{\mathbb{S}^n} {n\choose 1}\left[ \dfrac{(n-1)(n-2)}{2}\phi^{n-3}(\rho)\phi'^{3}(\rho) \right. \left.+K(1-\dfrac{3}{2}n)\phi^{n-1}(\rho)\phi'(\rho) \right] \rho^2 u^2\mathrm{d}A \notag\\
	&& + \int_{\mathbb{S}^n} {n \choose 1} \dfrac{n-1}{n} \phi^{n-3}(\rho)\phi'(\rho)\rho^2|\nabla u|^2 \mathrm{d}A  + O(\varepsilon)\|u\|_{L^2(\mathbb{S}^n)}^2 + O(\varepsilon) \|\nabla u\|_{L^2(\mathbb{S}^n)}^2\notag\\
	&&+Kn(\mathscr{A}_{-1}(\Omega)-\mathscr{A}_{-1}(\overline{B}_{\rho})).\notag \\
&=&\int_{\mathbb{S}^n} {n\choose 1}(n-1)\phi^{n-2}(\rho)\phi'^2(\rho) \rho u \mathrm{d}A \notag\\
	&& + \int_{\mathbb{S}^n} {n\choose 1}\left[ \dfrac{(n-1)(n-2)}{2}\phi^{n-3}(\rho)\phi'^{3}(\rho)\right.  +K(1-n)\phi^{n-1}(\rho)\phi'(\rho)\bigg] \rho^2 u^2\mathrm{d}A \notag\\
	&& + \int_{\mathbb{S}^n} {n \choose 1} \dfrac{n-1}{n} \phi^{n-3}(\rho)\phi'(\rho)\rho^2|\nabla u|^2 \mathrm{d}A \notag\\
	&& + O(\varepsilon)\|u\|_{L^2(\mathbb{S}^n)}^2 + O(\varepsilon) \|\nabla u\|_{L^2(\mathbb{S}^n)}^2. \label{a1ob}
\end{eqnarray}

\begin{lem}[Expression for $\mathscr{A}_{k}(\Omega)-\mathscr{A}_{k}(\overline{B}_{\rho})\ (k\geqslant 0)$]
	Let\ $\Omega\subset\mathbb{N}^{n+1}(K)\ (K=-1,1)$, $M=\partial\Omega=\{(\rho(1+u(x)),x):x\in\mathbb{S}^n\}$, where $u\in C^3(\mathbb{S}^n)$. Suppose $\|u\|_{W^{2,\infty}(\mathbb{S}^n)}<\varepsilon$, then for any\ $0\leqslant k\leqslant n$, there holds
	\begin{eqnarray}
		\mathscr{A}_{k}(\Omega)-\mathscr{A}_{k}(\overline{B}_{\rho})
		&=& \int_{\mathbb{S}^n}{n\choose k}(n-k)\phi^{n-k-1}(\rho)\phi'^{k+1}(\rho) \rho u \mathrm{d}A \notag\\
		&&+ \int_{\mathbb{S}^n} {n\choose k} \left[ \dfrac{(n-k)(n-k-1)}{2}\phi^{n-k-2}(\rho)\phi'^{k+2}(\rho)\right.\notag\\
		&&\ \ \ \ \ \ \ \ \ \ \ \ \ \ \left. -K\dfrac{(n-k)(k+1)}{2}\phi^{n-k}(\rho)\phi'^k(\rho) \right] \rho^2 u^2 \mathrm{d}A  \notag\\
		%%&&+ \int_{\mathbb{S}^n} {n\choose k} \dfrac{(n-k)(k+1)}{2n} \phi^{n-k-2}(1)\phi'^k(1) |\nabla u|^2 \mathrm{d}A \notag\\
		&&+ \int_{\mathbb{S}^n}  {n\choose k} \dfrac{(n-k)(k+1)}{2n}\phi^{n-k-2}(\rho)\phi'^k(\rho) \rho^2|\nabla u|^2  \mathrm{d}A \notag\\
		&&+O(\varepsilon)\|u\|_{L^2(\mathbb{S}^n)}^2 + O(\varepsilon) \|\nabla u\|_{L^2 (\mathbb{S}^n)}^2. \label{akob}
	\end{eqnarray}
\end{lem}
\begin{proof}
	By \eqref{a-1ob}, \eqref{a0ob} and \eqref{a1ob}, we have already checked the formulae for \eqref{akob} with $k=0,1$. It's easy to see that \eqref{akob} also satisfies \eqref{akomega}. Then by induction, we conclude our assertion.
\end{proof}

\section{Stability of Alexandrov-Fenchel inequalities}
\label{section4}
In this section, we prove Theorem \ref{mainresult}. %We will first use $\mathscr{A}_j(\Omega)= \mathscr{A}_{j}(\overline{B}_{\rho})\ (-1\leqslant j<k,\ 0\leqslant k\leqslant n-1)$ and $\mathrm{bar}(\Omega)=O$ to get a technical Poincar\'e-type estimate (Lemma \ref{lemma4.1}), and then estimate $\alpha^2(\Omega)$. Finally, we use the expression for $\delta_{k,j}(\Omega)$ to conclude the relation between $\delta_{k,j}(\Omega)\ (-1\leqslant j < k,\ 0\leqslant k\leqslant n-1)$ and $\alpha^2(\Omega)$.
For any fixed $0\leqslant k\leqslant n-1$, assuming $\mathscr{A}_j(\Omega)=\mathscr{A}_{j}(\overline{B}_{\rho})\ (-1\leqslant j < k)$ and $\mathrm{bar}(\Omega)=O$, we will get a Poincar\'e-type estimate for our proof. Before that, the following estimate which was proved in \cite{Bgelein2015ASQ,Bgelein2016AQI} is crucial. For the convenience of readers, we present it here:

\begin{lem}[\cite{Bgelein2015ASQ,Bgelein2016AQI}]
	\label{a_1^2-lem}
	Let $\Omega$ be a nearly spherical domain enclosed by $M=\{(\rho(1+u(x)),x):x\in\mathbb{S}^n\}$ in $\mathbb{N}^{n+1}(K)\ (K=-1,1)$, and suppose $u$ can be represented as
	\begin{equation*}
		u=\sum\limits_{k=0}^{\infty} a_kY_k,
	\end{equation*}
	where $\{Y_k\}_{k=0}^{\infty}$ corresponds the spherical harmonics which forms an orthonormal basis for $L^2(\mathbb{S}^n)$. If $$\mathrm{bar}(\Omega) = O,$$
	then
	\begin{equation}
		a_1^2 = O(\varepsilon)\|u\|_{L^2(\mathbb{S}^n)}^2.  \label{a1est}
	\end{equation}  
\end{lem}
\begin{proof}
	Note that the eigenvalues of the Laplace-Beltrami operator $\Delta$ on\ $\mathbb{S}^n$ are\ $\lambda_k=-k(n+k-1)\ (k=0,1,\cdots)$. Let\ $\{Y_k\}_{k=0,1,\cdots}$ be the eigenfunction corresponding to $\lambda_k$ which constitutes an orthonormal basis in\ $L^2(\mathbb{S}^n)$. We can take
	\begin{equation}
		Y_0=\dfrac{1}{\sqrt{\mathrm{Area}(\mathbb{S}^n)}},\ \ Y_1=\dfrac{\sqrt{n+1}}{\sqrt{\mathrm{Area}(\mathbb{S}^n)}} x\cdot v,
		\label{y0y1}
	\end{equation}
	where $v=(v_1,v_2,\cdots,v_{n+1})\in\mathbb{S}^n$, $x\cdot v=\sum\limits_{l=1}^{n+1}x_lv_l$. 
	
	From Definition \ref{defofbary}, $\mathrm{bar}(\Omega)=O$ is characterized by the necessary condition
	\begin{equation}
		0 = \int_{\Omega} \left. \nabla^{\mathbb{N}^{n+1}(K)}_p \left( d_K^2(y,p) \right) \right|_{p=O} \mathrm{d}\mu_K(y)=\int_{\Omega} \left. 2 d_K(y,O) \nabla^{\mathbb{N}^{n+1}(K)}_p[d_K(y,p)]\right|_{p=O} \mathrm{d}\mu_K(y).
	\end{equation}
	By using the warped product metric for $\mathbb{N}^{n+1}(K)\ (K=-1,1)$, we have
	\begin{equation}
		\int_0^{\rho} \int_{\mathbb{S}^n} r(1+u(x))^2\phi^n(r(1+u(x))) x_l \mathrm{d}A\mathrm{d}r=0,\ \ l=1,2,\cdots,n+1. \label{barycenter}
	\end{equation}
	By Taylor expansion of $r(1+u)^2\phi^n(r(1+u))$ at $u=0$, we have
	\begin{equation*}
		r(1+u)^2\phi^n(r(1+u)) = r\phi^n(r) + \dfrac{\mathrm{d}}{\mathrm{d}r}\left( r^2\phi^n(r) \right) u + o(u), 
	\end{equation*}
	then insert it into (\ref{barycenter}) and use $\|u\|_{W^{2,\infty}(\mathbb{S}^n)}<\varepsilon$, we get that for any $1\leqslant l\leqslant n+1$,
	\begin{equation}
		0 = \int_0^{\rho} r\phi^n(r) \mathrm{d}r \int_{\mathbb{S}^n}  x_l \mathrm{d}A + \rho^2\phi^n(\rho) \int_{\mathbb{S}^n} ux_l \mathrm{d}A + O(\varepsilon)\|u\|_{L^2(\mathbb{S}^n)}. \label{baromega}
	\end{equation}
	Note that $\mathrm{bar}(\overline{B}_{\rho})=O$, by taking $u=0$ in (\ref{barycenter}) we have
	\begin{equation}
		\int_{\mathbb{S}^n} x_l \mathrm{d}A=0,\ \ l=1,2,\cdots,n+1. \label{barb}
	\end{equation}
	Combining (\ref{baromega}) with (\ref{barb}), we know that
	\begin{equation}
		\int_{\mathbb{S}^n} ux_l \mathrm{d}A = O(\varepsilon)\|u\|_{L^2(\mathbb{S}^n)},\ \ l=1,2,\cdots,n+1.
	\end{equation}
	Consequently, by (\ref{y0y1}),
	\begin{equation}
		a_1=\int_{\mathbb{S}^n} uY_1 \mathrm{d}A = \dfrac{\sqrt{n+1}}{\sqrt{\mathrm{Area}(\mathbb{S}^n)}} \sum\limits_{l=1}^{n+1} v_l \int_{\mathbb{S}^n} ux_l \mathrm{d}A = O(\varepsilon) \|u\|_{L^2(\mathbb{S}^n)}.
	\end{equation}
\end{proof}
\begin{lem}[Poincar\'e-type Estimate]
	\label{lemma4.1}
	Let $M=\{(\rho(1+u(x)),x):x\in\mathbb{S}^n\}$, where $u\in C^3(\mathbb{S}^n)$, and there exists small $\varepsilon>0$ such that $\|u\|_{W^{2,\infty}(\mathbb{S}^n)}<\varepsilon$. Let $\Omega$ be the domain enclosed by $M$ in $\mathbb{N}^{n+1}(K)\ (K=-1,1)$ satisfying $\mathrm{bar}(\Omega)=O$.
	\begin{enumerate}[label={(\arabic*)},parsep=0cm,itemsep=0cm,topsep=0cm]
		\item If $\mathscr{A}_{-1}(\Omega) = \mathscr{A}_{-1}(\overline{B}_{\rho})$, then
		\begin{equation}
			\|\nabla u\|_{L^2(\mathbb{S}^n)}^2\geqslant 2(n+1)\|u\|_{L^2(\mathbb{S}^n)}^2 + O(\varepsilon)\|u\|_{L^2(\mathbb{S}^n)}^2.  \label{egenest}
		\end{equation}
		\item If $\mathscr{A}_{j}(\Omega) = \mathscr{A}_{j}(\overline{B}_{\rho})$, for any $0\leqslant j< k$, where $1\leqslant k\leqslant n$, then
		\begin{equation}
			\|\nabla u\|_{L^2(\mathbb{S}^n)}^2\geqslant 2(n+1)\|u\|_{L^2(\mathbb{S}^n)}^2 + O(\varepsilon)\|u\|_{L^2(\mathbb{S}^n)}^2+O(\varepsilon)\|\nabla u\|_{L^2(\mathbb{S}^n)}^2. \label{egenest22}
		\end{equation}
	\end{enumerate}
\end{lem}
\begin{proof}
	Consider the Fourier expansion $u=\sum\limits_{k=0}^{\infty}a_kY_k$ as in Lemma \ref{a_1^2-lem}, 
	then
	$$\|u\|_{L^2(\mathbb{S}^n)}^2=\sum\limits_{k=0}^{\infty}a_k^2,\ \ \|\nabla u\|_{L^2(\mathbb{S}^n)}^2 = \sum\limits_{k=0}^{\infty} |\lambda_k| a_k^2.$$
	
	When\ $k\geqslant 2$, we have that $|\lambda_k|=k(n+k-1)\geqslant 2(n+1)$, then
	\begin{eqnarray}
		\|\nabla u\|_{L^2(\mathbb{S}^n)}^2 &=& \sum\limits_{k=0}^{\infty} |\lambda_k| a_k^2=\sum\limits_{k=2}^{\infty} |\lambda_k| a_k^2 + na_1^2 \geqslant 2(n+1)\sum\limits_{k=2}^{\infty} a_k^2 +na_1^2 \notag\\
		&=& 2(n+1)\sum\limits_{k=0}^{\infty} a_k^2 - (n+2) a_1^2 - 2(n+1) a_0^2 \notag\\
		&=& 2(n+1) \|u\|_{L^2(\mathbb{S}^n)}^2  - (n+2) a_1^2 - 2(n+1) a_0^2 \label{nablauest}
	\end{eqnarray}
	
	Now $a_1^2$ has been estimated in Lemma \ref{a_1^2-lem}. Note that $a_0= \dfrac{1}{\sqrt{\mathrm{Area}(\mathbb{S}^n)}}\displaystyle\int_{\mathbb{S}^n} u \mathrm{d}A$, 
	%and it has been proved in \cite{Bgelein2015ASQ} and \cite{Bgelein2016AQI} that  $\mathrm{bar}(\Omega)=0$ means $a_1^2=O(\varepsilon)\|u\|_{L^2(\mathbb{S}^n)}^2$. 
	we estimate $a_0^2$ in two cases.
	
	(1) Assume that\ $\mathrm{Vol}(\Omega)=\mathrm{Vol}(\overline{B}_{\rho})$, by (\ref{a-1ob}) we get
	\begin{equation}
		\int_{\mathbb{S}^n} u\mathrm{d}A = -\dfrac{n}{2}\dfrac{\phi'(\rho)}{\phi(\rho)}\rho \int_{\mathbb{S}^n} u^2 \mathrm{d}A + O(\varepsilon) \|u\|_{L^2(\mathbb{S}^n)}^2,  \label{volumeq}
	\end{equation}
	therefore, we get
	\begin{equation}
		a_0^2=O(\varepsilon)\|u\|_{L^2(\mathbb{S}^n)}^2.  \label{a0est-1}
	\end{equation}
	
	(2)  Assume that $\mathscr{A}_j(\Omega) = \mathscr{A}_j(\overline{B}_{\rho})\ (0\leqslant j<k)$, by (\ref{akob}), we have
	\begin{eqnarray}
		\int_{\mathbb{S}^n} u \mathrm{d}A &=& -\int_{\mathbb{S}^n} \left[\dfrac{n-j-1}{2}\dfrac{\phi'(\rho)}{\phi(\rho)} -K\dfrac{j+1}{2} \dfrac{\phi(\rho)}{\phi'(\rho)}  \right]\rho u^2 \mathrm{d}A\notag\\
		&&-\int_{\mathbb{S}^n} \dfrac{j+1}{2n}\dfrac{1}{\phi(\rho)\phi'(\rho)}\rho |\nabla u|^2 \mathrm{d}A + O(\varepsilon)\|u\|_{L^2(\mathbb{S}^n)}^2  + O(\varepsilon)\|\nabla u\|_{L^2(\mathbb{S}^n)}^2. \label{ujjj}
	\end{eqnarray}
	Thus we have
	\begin{equation}
		a_0^2=O(\varepsilon)\|u\|_{L^2(\mathbb{S}^n)}^2+O(\varepsilon)\|\nabla u\|_{L^2(\mathbb{S}^n)}^2.  \label{a0estj}
	\end{equation}  
	
	Combine (\ref{nablauest}), (\ref{a1est}), (\ref{a0est-1}) and (\ref{a0estj}), we get the estimate (\ref{egenest}) and (\ref{egenest22}).
\end{proof}

Now, Let's estimate $\alpha^2(\Omega)$ by definition.
\begin{lem}[Estimation for $\alpha^2(\Omega)$]
	Let $\Omega$ be a nearly spherical domain in $\mathbb{N}^{n+1}(K)\ (K=-1,1)$ which is enclosed by $M=\{ (\rho(1+u(x)),x): x\in\mathbb{S}^n \}$. Suppose $u\in C^3(\mathbb{S}^n)$ and there exists $\varepsilon>0$ such that $\|u\|_{W^{2,\infty}(\mathbb{S}^n)}<\varepsilon$, then
	\begin{equation}
		\alpha^2(\Omega) \leqslant \dfrac{1}{n^2} \mathrm{Area}(\mathbb{S}^n)\phi^{2n}(\rho) \rho^2 \|\nabla u\|_{L^2(\mathbb{S}^n)}^2 + O(\varepsilon)\|\nabla u\|_{L^2(\mathbb{S}^n)}^2.
	\end{equation} 
\end{lem}
\begin{proof}
	Let $\overline{B}_{\Omega}$ be a geodesic ball with geodesic radius $R$ centered at $O$ in $\mathbb{N}^{n+1}(K)\ (K=-1,1)$ such that
	\begin{equation}
		\mathrm{Vol}(\Omega) = \mathrm{Vol}(\overline{B}_{\Omega}) = \mathrm{Area(\mathbb{S}^n)}\int_{0}^R \phi^n(r)\mathrm{d}r.
	\end{equation}
	Thus,
	\begin{eqnarray}
		&&\mathrm{Vol}(\Omega\Delta\overline{B}_{\Omega}) \notag\\
		&=& \int_{\mathbb{S}^n} \left| \int_0^{\rho(1+u)}\phi^n(r)\mathrm{d}r -\int_0^R \phi^n(r) \mathrm{d}r \right| \mathrm{d}A  \notag\\
		&=& \int_{\mathbb{S}^n} \left| \int_0^{\rho(1+u)}\phi^n(r)\mathrm{d}r -\dfrac{1}{\mathrm{Area(\mathbb{S}^n)}}\int_{\mathbb{S}^n}\left(\int_0^{\rho(1+u)} \phi^n(r) \mathrm{d}r\right) \mathrm{d}A   \right| \mathrm{d}A  \notag\\
		&=& \left\| \int_0^{\rho(1+u)}\phi^n(r)\mathrm{d}r -\dfrac{1}{\mathrm{Area(\mathbb{S}^n)}}\int_{\mathbb{S}^n}\left(\int_0^{\rho(1+u)} \phi^n(r) \mathrm{d}r\right) \mathrm{d}A \right\|_{L^1(\mathbb{S}^n)}\notag\\
		&\leqslant& \left( \mathrm{Area}(\mathbb{S}^n) \right)^{\frac{1}{2}} \left\| \int_0^{\rho(1+u)}\phi^n(r)\mathrm{d}r -\dfrac{1}{\mathrm{Area(\mathbb{S}^n)}}\int_{\mathbb{S}^n}\left(\int_0^{\rho(1+u)} \phi^n(r) \mathrm{d}r\right) \mathrm{d}A \right\|_{L^2(\mathbb{S}^n)}, \label{volarea}
	\end{eqnarray}
	where in (\ref{volarea}) we used Hölder's inequality. By Poincaré's inequality,
	\begin{eqnarray}
		&&\left\| \int_0^{\rho(1+u)}\phi^n(r)\mathrm{d}r -\dfrac{1}{\mathrm{Area(\mathbb{S}^n)}}\int_{\mathbb{S}^n}\left(\int_0^{\rho(1+u)} \phi^n(r) \mathrm{d}r\right) \mathrm{d}A \right\|_{L^2(\mathbb{S}^n)}   \notag\\
		&\leqslant& \dfrac{1}{n} \left\| \nabla \left( \int_0^{\rho(1+u)}\phi^n(r)\mathrm{d}r \right) \right\|_{L^2(\mathbb{S}^n)}=\dfrac{1}{n} \left\| \phi^{n}(\rho(1+u)) \rho \nabla u  \right\|_{L^2(\mathbb{S}^n)} \notag\\
		&\leqslant& \dfrac{1}{n} \rho \| \phi^n(\rho(1+u)) \|_{L^{\infty}(\mathbb{S}^n)} \| \nabla u \|_{L^2(\mathbb{S}^n)}.
	\end{eqnarray}
	Thus
	\begin{equation}
		\mathrm{Vol}(\Omega\Delta\overline{B}_{\Omega}) \leqslant \dfrac{1}{n} \rho  \left( \mathrm{Area}(\mathbb{S}^n) \right)^{\frac{1}{2}} \| \phi^n(\rho(1+u)) \|_{L^{\infty}(\mathbb{S}^n)}\|\nabla u\|_{L^2(\mathbb{S}^n)}.
	\end{equation}
	By Taylor expansion of $\phi^n(\rho(1+u))$ at $u=0$,
	\begin{eqnarray}
		\phi^n(\rho(1+u))&=&\phi^n(\rho) + n\phi^{n-1}(\rho)\phi'(\rho) \rho u \notag\\
		&&+ \dfrac{1}{2}\left[ n(n-1)\phi^{n-2}(\rho)\phi'^2(\rho)-Kn\phi^{n}(\rho) \right]\rho^2 u^2 +o(u^2),
	\end{eqnarray}
	and notice the condition $\|u\|_{C^2(\mathbb{S}^n)}<\varepsilon$, we have
	\begin{equation}
		\dfrac{\|\phi^n (\rho(1+u)) \|_{L^{\infty}(\mathbb{S}^n)}}{\phi^n(\rho)} = 1+ O(\varepsilon).  \label{phioe}
	\end{equation}
	Therefore,
	\begin{equation}
		\mathrm{Vol}(\Omega\Delta\overline{B}_{\Omega}) \leqslant \left(  \dfrac{1}{n}\rho \left( \mathrm{Area}(\mathbb{S}^n) \right)^{\frac{1}{2}} \phi^n(\rho) + O(\varepsilon)  \right) \| \nabla u\|_{L^2(\mathbb{S}^n)}.
	\end{equation}
	Thus
	\begin{equation}
		\alpha^2(\Omega) \leqslant \mathrm{Vol}(\Omega\Delta\overline{B}_{\Omega})^2 \leqslant \dfrac{1}{n^2}\rho^2 \mathrm{Area}(\mathbb{S}^n)  \phi^{2n}(\rho)  \| \nabla u\|_{L^2(\mathbb{S}^n)}^2 + O(\varepsilon) \| \nabla u\|_{L^2(\mathbb{S}^n)}^2.
	\end{equation}
	We get the conclusion.
\end{proof}
From this lemma, we know that
\begin{equation}
	\|\nabla u\|_{L^2 (\mathbb{S}^n)}^2 \geqslant \left( \dfrac{n^2}{\mathrm{Area}(\mathbb{S}^n)\phi^{2n}(\rho)\rho^2}+O(\varepsilon) \right) \alpha^2(\Omega).  \label{nablau}
\end{equation}

Now we prove Theorem \ref{mainresult}. 
\begin{proof}[Proof of Theorem \ref{mainresult}]
	By $\mathscr{A}_j(\Omega) = \mathscr{A}_j(\overline{B}_{\rho})\ (\mathrm{fixed}-1\leqslant j<k)$, that is when $j=-1$, substitute (\ref{volumeq}) into (\ref{akob}), and when $0\leqslant j<k$, substitute (\ref{ujjj}) into (\ref{akob}), after a direct computation and use (\ref{hd}) we have
	\begin{eqnarray}
		\mathscr{A}_k(\Omega) - \mathscr{A}_k(\overline{B}_{\rho}) &=& -\int_{\mathbb{S}^n} {n\choose k}\dfrac{(n-k)(k-j)}{2}\phi^{n-k-2}(\rho)\phi'^k(\rho) \rho^2 u^2 \mathrm{d}A\notag \\
		&&+ \int_{\mathbb{S}^n} {n\choose k} \dfrac{(n-k)(k-j)}{2n}\phi^{n-k-2}(\rho)\phi'^k(\rho) \rho^2|\nabla u|^2 \mathrm{d}A\notag\\
		&&+O(\varepsilon)\|u\|_{L^2(\mathbb{S}^n)}^2 + O(\varepsilon)\|\nabla u\|_{L^2(\mathbb{S}^n)}^2.
	\end{eqnarray}
	
	Note that the coefficient of $\displaystyle\int_{\mathbb{S}^n} |\nabla u|^2 \mathrm{d}A$ is positive, by estimate (\ref{egenest}) and (\ref{egenest22}) for $\|\nabla u\|_{L^2(\mathbb{S}^n)}^2$, we have
	\begin{eqnarray}
		\mathscr{A}_k(\Omega) - \mathscr{A}_k(\overline{B}_{\rho}) &\geqslant& \left({n\choose k}\dfrac{(n-k)(k-j)}{4n}\phi^{n-k-2}(\rho)\phi'^{k}(\rho)\rho^2+O(\varepsilon)\right) \|\nabla u\|_{L^2(\mathbb{S}^n)}^2\notag\\
		&&+\left({n\choose k}\dfrac{(n-k)(k-j)}{2n}\phi^{n-k-2}(\rho)\phi'^{k}(\rho)\rho^2+O(\varepsilon)\right) \| u\|_{L^2(\mathbb{S}^n)}^2\notag \\
		&&+O(\varepsilon) \|u\|_{L^2(\mathbb{S}^n)}^2+O(\varepsilon)\|\nabla u\|_{L^2(\mathbb{S}^n)}^2 \notag\\
		&\geqslant &{n\choose k}\dfrac{(n-k)(k-j)}{4n}\phi^{n-k-2}(\rho)\phi'^{k}(\rho) \rho^2 \|\nabla u\|_{L^2(\mathbb{S}^n)}^2.  \label{akest}
	\end{eqnarray}
	
	After combining (\ref{akest}) with (\ref{nablau}), we finally conclude that
	\begin{eqnarray}
		\delta_{k,j}(\Omega)&=&\mathscr{A}_k(\Omega) - \mathscr{A}_k(\overline{B}_{\rho})\notag\\
		&\geqslant& \left({n\choose k}\dfrac{n(n-k)(k-j)}{4}\dfrac{\phi^{n-k-2}(\rho)\phi'^{k}(\rho)}{\phi^{2n}(\rho)\mathrm{Area}(\mathbb{S}^n)}+O(\varepsilon)\right)\alpha^2(\Omega)\notag \\
		&=& \left( \dfrac{n(n-k)(k-j)}{4\mathrm{Area}(\mathbb{S}^n)}{n\choose k} \dfrac{\phi'^k(\rho)}{\phi^{n+k+2}(\rho)} + O(\varepsilon) \right)\alpha^2(\Omega).
		%&=& \left( \dfrac{n(n-k)(k-j)}{4\phi^2(\rho)}\dfrac{{n\choose k}\phi^{n-k}(\rho)\phi'^k(\rho)\mathrm{Area}(\mathbb{S}^n)}{\mathscr{A}_0^2(\overline{B}_{\rho})} +O(\varepsilon) \right)\alpha^2(\Omega).
	\end{eqnarray}
\end{proof}

\section{Stability of weighted quermassintegral inequalities}
\label{section5}

In this section, we discuss the stability of geometric inequalities involving weighted curvature integrals and quermassintegrals for nearly spherical sets in $\mathbb{R}^{n+1}$ and $\mathbb{H}^{n+1}$. 

\subsection{Stability of weighted quermassintegral inequalities in $\mathbb{R}^{n+1}$}
We establish the inequality in $\mathbb{R}^{n+1}$ which states that for any fixed $0\leqslant k\leqslant n$, if $\mathscr{A}_l(\Omega)=\mathscr{A}_l (B)\ (-1\leqslant l<k)$, where $B$ is the unit ball in $\mathbb{R}^{n+1}$, then
\begin{equation*}
	\dfrac{\displaystyle\int_M \Phi\sigma_k(\kappa)\mathrm{d}\mu_g - \int_{\partial B} \Phi\sigma_k(\kappa) \mathrm{d}\mu_g}{\displaystyle\int_{\partial B} \Phi\sigma_k(\kappa) \mathrm{d}\mu_g} \geqslant C\overline{\alpha}^2(\Omega),
\end{equation*}
where $C>0$ is a constant independent of $\Omega$, $\Phi(r)=\displaystyle\int_0^{r} r\mathrm{d}r=\dfrac{1}{2}r^2$ is defined in (\ref{Phi}), and $\overline{\alpha}(\Omega)$ is defined in (\ref{alphaomegaeuclidean}). It is the quantitative version of (\ref{weightedinR}).
\begin{proof}[Proof of Theorem \ref{mainresult2}]
		Taking $\rho=1,\ K=0,\ \phi(r)=r$ in \eqref{PhiinN}, we get that for any nearly spherical set $M=\{(1+u(x)),x):x\in\mathbb{S}^n\}$ in $\mathbb{R}^{n+1}$,
	\begin{eqnarray}
		\int_M \Phi\sigma_k(\kappa)\mathrm{d}\mu_g &=& \dfrac{1}{2}\int_{\mathbb{S}^n} {n\choose k} \mathrm{d}A + \dfrac{1}{2}\int_{\mathbb{S}^n} {n\choose k}(n-k+2)u \mathrm{d}A \notag\\
		&&+ \dfrac{1}{2} \int_{\mathbb{S}^n} {n\choose k}  \dfrac{(n-k+2)(n-k+1)}{2}u^2 \mathrm{d}A \notag\\
		&&  + \dfrac{1}{2} \int_{\mathbb{S}^n} {n\choose k}  \dfrac{(n-k)(k+1)+4k}{2n} |\nabla u|^2 \mathrm{d}A \notag\\
		&&+ O(\varepsilon)\|u\|_{L^2(\mathbb{S}^n)}^2 + O(\varepsilon)\|\nabla u\|_{L^2(\mathbb{S}^n)}^2. \label{PhiinR}
	\end{eqnarray}
	From $\mathscr{A}_{-1}(\Omega)=\mathscr{A}_{-1}(B)$, we have
	\begin{equation}
		\label{a-1inR}
		\int_{\mathbb{S}^n} u \mathrm{d}A = -\int_{\mathbb{S}^n}\dfrac{n}{2}u^2 \mathrm{d}A + O(\varepsilon)\|u\|_{L^2(\mathbb{S}^n)}^2,
	\end{equation}
	and from $\mathscr{A}_j(\Omega) = \mathscr{A}_j(B)\ (0\leqslant j < k)$, we have
	\begin{equation}
		\label{ajinR}
		\int_{\mathbb{S}^n} u \mathrm{d}A = -\int_{\mathbb{S}^n} \dfrac{n-j-1}{2} u^2 \mathbb{d}A - \int_{\mathbb{S}^n} \dfrac{j+1}{2n} |\nabla u|^2 \mathrm{d}A + O(\varepsilon) \|u\|_{L^2(\mathbb{S}^n)}^2 + O(\varepsilon) \|\nabla u\|_{L^2(\mathbb{S}^n)}^2.
	\end{equation}
    These are proved in Proposition 4.3 and Lemma 5.2 of \cite{VanBlargan2022QuantitativeQI}.
	Then by substituting (\ref{a-1inR}) in (\ref{PhiinR}) when $j=-1$ and substituting (\ref{ajinR}) in (\ref{PhiinR}) when $0\leqslant j<k$,
	\begin{eqnarray}
		\int_M \Phi\sigma_k(\kappa) \mathrm{d}\mu_g - \dfrac{1}{2} \int_{\mathbb{S}^n} {n \choose k} \mathrm{d}A &=& \dfrac{1}{2}\int_{\mathbb{S}^n} {n\choose k}\dfrac{(n-k+2)(j-k+2)}{2} u^2 \mathrm{d}A \notag\\
		&& + \dfrac{1}{2} \int_{\mathbb{S}^n} {n\choose k}\dfrac{(n-k+2)(k-j)+2k-2}{2n} |\nabla u|^2 \mathrm{d}A \notag\\
		&& + O(\varepsilon)\|u\|_{L^2(\mathbb{S}^n)}^2 + O(\varepsilon)\|\nabla u\|_{L^2(\mathbb{S}^n)}^2.  \label{phisigmak}
	\end{eqnarray}
	We remark that the coefficient of $\displaystyle\int_{\mathbb{S}^n} |\nabla u|^2 \mathrm{d}A$ in (\ref{phisigmak}) is positive. Therefore, under the condition $\mathrm{bar}(\Omega)=O$, substitute
	\begin{equation*}
		\|\nabla u\|_{L^2(\mathbb{S}^n)}^2 \geqslant 2(n+1) \|u\|_{L^2(\mathbb{S}^n)}^2 + O(\varepsilon)\| u \|^2_{L^2(\mathbb{S}^n)}
	\end{equation*}
	(when $j=-1$) and 
	\begin{equation*}
		\|\nabla u\|_{L^2(\mathbb{S}^n)}^2 \geqslant 2(n+1) \|u\|_{L^2(\mathbb{S}^n)}^2 + O(\varepsilon)\|u\|^2_{L^2(\mathbb{S}^n)} + O(\varepsilon)\|\nabla u\|_{L^2(\mathbb{S}^n)}^2
	\end{equation*}
	(when $0\leqslant j <k $) in (\ref{phisigmak}), we get
	\begin{eqnarray}
		\displaystyle\int_M \Phi\sigma_k(\kappa) \mathrm{d}\mu_g - \int_{\partial B}\Phi\sigma_k(\kappa)\mathrm{d}\mu_g &\geqslant& {n\choose k}\dfrac{n-k+2}{2}  \|u\|_{L^2(\mathbb{S}^n)}^2 \notag\\
		&&+ {n\choose k}\dfrac{(n-k+2)(k-j)+2k-2}{8n} \|\nabla u\|_{L^2(\mathbb{S}^n)}^2 \notag\\
		&& + O(\varepsilon)\|u\|_{L^2(\mathbb{S}^n)}^2 + O(\varepsilon)\|\nabla u\|_{L^2(\mathbb{S}^n)}^2.\notag\\
		&\geqslant& {n\choose k} \dfrac{(n-k+2)(k-j)+2k-2}{8n}\|\nabla u\|_{L^2(\mathbb{S}^n)}^2.
	\end{eqnarray}
	Combining the estimate proved in Lemma 5.3 of \cite{VanBlargan2022QuantitativeQI}:
	\begin{equation}
		\overline{\alpha}^2(\Omega) %\leqslant \dfrac{|\Omega \Delta B_{\Omega}|^2}{|B_{\Omega}|^2} 
		\leqslant \dfrac{(n+1)^2}{n^2\mathrm{Area}(\mathbb{S}^n)}\|\nabla u\|^2_{L^2(\mathbb{S}^n)} + O(\varepsilon)\|\nabla u\|_{L^2(\mathbb{S}^n)}^2,
	\end{equation}
	we obtain the inequality
	\begin{eqnarray}
		\dfrac{\displaystyle\int_{M}\Phi\sigma_k(\kappa)\mathrm{d}\mu_g-\int_{\partial B}\Phi\sigma_k(\kappa)\mathrm{d}\mu_g}{\displaystyle\int_{\partial B}\Phi\sigma_k(\kappa)\mathrm{d}\mu_g} %&=& \dfrac{\displaystyle\int_{M}\Phi\sigma_k(\kappa)\mathrm{d}\mu_g-\dfrac{1}{2}{n\choose k}\mathrm{Area}(\mathbb{S}^n)}{\displaystyle\dfrac{1}{2}{n\choose k}\mathrm{Area}(\mathbb{S}^n)} \notag\\
		&\geqslant& \dfrac{(n-k+2)(k-j)+2k-2}{4n\mathrm{Area}(\mathbb{S}^n)}\|\nabla u\|^2_{L^2(\mathbb{S}^n)} \notag\\
		&\geqslant& \left( \dfrac{n \left((n-k+2)(k-j)+2k-2\right)}{4(n+1)^2}+O(\varepsilon) \right) \overline{\alpha}^2(\Omega) \notag
	\end{eqnarray}
	as desired.
\end{proof}

\subsection{Stability of weighted quermassintegral inequalities in $\mathbb{H}^{n+1}$}
We are going to establish the stability inequality in $\mathbb{H}^{n+1}$ which states that for any fixed $0\leqslant k\leqslant n$, $-1\leqslant j<k$, if $\mathscr{A}_j(\Omega)=\mathscr{A}_j(\overline{B}_{\rho})$ holds, then
\begin{equation*}
	\int_M \Phi\sigma_k(\kappa)\mathrm{d}\mu_g - \int_{\partial \overline{B}_{\rho}}\Phi\sigma_k(\kappa) \mathrm{d}\mu_g \geqslant C\alpha^2(\Omega),
\end{equation*}
where $C>0$ is a constant independent of $\Omega$, and $\Phi$ is defined in (\ref{Phi}). This is the quantitative version of (\ref{weightedH}).

\begin{proof}[Proof of \eqref{hqq} in Theorem \ref{mainresult3}]
	Substituting (\ref{volumeq}) (when $j=-1$) and (\ref{ujjj}) (when $0\leqslant j <k$) in (\ref{PhiinN}), using (\ref{hd}) we get
	\begin{eqnarray}
		&&\int_M \Phi\sigma_k(\kappa)  \mathrm{d}\mu_g - \int_{\partial \overline{B}_{\rho}} \Phi\sigma_k(\kappa) \mathrm{d}\mu_g \notag\\
		&=& \int_{\mathbb{S}^n} {n\choose k} \left\lbrace \left[ \dfrac{(n-k)(j-k)}{2}\phi^{n-k-2}(\rho)\phi'^k(\rho) + K\dfrac{k(k-j-2)}{2} \phi^{n-k}(\rho) \phi'^{k-2}(\rho) \right]\Phi(\rho)  \right.  \notag\\
		&&\ \ \ \ \ \ \ \ \ \ \ \ \ \left. + \dfrac{n+1}{2}\phi^{n-k}(\rho)\phi'^{k+1}(\rho) + (\dfrac{j+1}{2}-k)\phi^{n-k}(\rho)\phi'^{k-1}(\rho) \right\rbrace \rho^2 u^2 \mathrm{d}A \notag\\
		&& + \int_{\mathbb{S}^n} {n\choose k} \left\lbrace \left[ \dfrac{(n-k)(k-j)}{2n}\phi^{n-k-2}(\rho)\phi'^k(\rho) - K\dfrac{k(k-j-2)}{2n}\phi^{n-k}(\rho)\phi'^{k-2}(\rho) \right]\Phi(\rho) \right. \notag\\
		&&\ \ \ \ \ \ \ \ \ \ \ \ \ \ \ \ \left. + \dfrac{2k-j-1}{2n} \phi^{n-k}(\rho) \phi'^{k-1}(\rho) \right\rbrace \rho^2 |\nabla u|^2 \mathrm{d}A \notag \\
		&& + O(\varepsilon) \|u\|_{L^2(\mathbb{S}^n)}^2 + O(\varepsilon) \|\nabla u\|_{L^2(\mathbb{S}^n)}^2  \notag\\
		&=&  C_1(n,k,j,\rho)  \|u\|_{L^2(\mathbb{S}^n)}^2 + C_2(n,k,j,\rho) \|\nabla u\|_{L^2(\mathbb{S}^n)}^2+ O(\varepsilon)\|u\|_{L^2(\mathbb{S}^n)}^2 \notag\\
		&& + O(\varepsilon) \|\nabla u\|_{L^2(\mathbb{S}^n)}^2.  \label{down}
	\end{eqnarray}
	
	Notice that when $K=-1$, $\phi'(\rho)=\cosh\rho\geqslant 1$, thus the coefficient of $\displaystyle\int_{\mathbb{S}^n} |\nabla u|^2 \mathrm{d}A$, is positive for $-1\leqslant j<k$. Indeed, by using $\Phi(\rho) = K(1-\phi'(\rho))$, we can compute that
	\begin{eqnarray}
		C_2(n,k,j,\rho) %&=& {n\choose k}\left[ \dfrac{(n-k)(k-j)}{2n} \phi^{n-k-2}(\rho)\phi'^k(\rho)\Phi(\rho)  \right.\notag\\
		%&&\left.\ \ + \left( \dfrac{k(k-j)-j-1}{2n}\phi'(\rho) -\dfrac{k(k-j-2)}{2n} \right) \phi^{n-k}(\rho) \phi'^{k-2}(\rho) \right]\notag\\
		&=& {n\choose k}\phi^{n-k-2}(\rho)\phi'^{k-2}(\rho)\left[ \dfrac{(n-k)(k-j)}{2n}\Phi(\rho) \right.\notag \\
		&&\ \ \ \ \left.+ \phi^2(\rho)\left( \dfrac{n(k-j)-j-1}{2n}\phi'(\rho) - \dfrac{n(k-j)-2k}{2n} \right) \right]\rho^2 > 0. \label{C2coe}
	\end{eqnarray}
	After using the Poincar\'e-type estimate (\ref{egenest}) and (\ref{egenest22}) in (\ref{down}), we get
	\begin{eqnarray}
		&&\int_M \Phi\sigma_k(\kappa) \mathrm{d}\mu_g - \int_{\partial\overline{B}_{\rho}} \Phi\sigma_k(\kappa) \mathrm{d}\mu_g \notag\\
		&\geqslant& \left( C_1(n,k,j,\rho) + C_2(n,k,j,\rho) n \right) \|u\|_{L^2(\mathbb{S}^n)}^2  + \dfrac{1}{2} C_2(n,k,j,\rho) \|\nabla u\|_{L^2(\mathbb{S}^n)}^2  \notag \\
		&&+ O(\varepsilon)\|u\|_{L^2(\mathbb{S}^n)}^2 + O(\varepsilon) \|\nabla u\|_{L^2(\mathbb{S}^n)}^2.\notag\\
		&\geqslant& \dfrac{1}{2} C_2(n,k,j,\rho) \|\nabla u\|_{L^2(\mathbb{S}^n)}^2\notag\\
		&\geqslant& {n\choose k} \dfrac{(n-k)(k-j)}{4n}\phi^{n-k-2}(\rho)\phi'^{k-2}(\rho)\Phi(\rho)\rho^2 \|\nabla u\|_{L^2(\mathbb{S}^n)}^2, \label{philow}
	\end{eqnarray}
	where we used the fact that $C_1(n,k,j,\rho) + C_2(n,k,j,\rho)n\geqslant 0$ in the second inequality by a direct computation.
	
	Finally, combining (\ref{philow}) with the estimate (\ref{nablau}) for $\|\nabla u\|_{L^2(\mathbb{S}^n)}^2$, we obtain the inequality
	\begin{eqnarray}
		&&\int_M \Phi\sigma_k(\kappa) \mathrm{d}\mu_g - \int_{\partial\overline{B}_{\rho}} \Phi\sigma_k(\kappa) \mathrm{d}\mu_g \notag \\
		&\geqslant& \left( {n\choose k} \dfrac{n(n-k)(k-j)}{4}\dfrac{\phi'^{k-2}(\rho)\Phi(\rho)}{\mathrm{Area}(\mathbb{S}^n)\phi^{n+k+2}(\rho)} +O(\varepsilon) \right)\alpha^2(\Omega). \notag
		%&=& \left( \dfrac{n(n-k)(k-j)}{4\phi^2(\rho)\phi'^2(\rho)}\dfrac{{n\choose k}\phi^{n-k}(\rho)\phi'^k(\rho)\Phi(\rho)\mathrm{Area}(\mathbb{S}^n)}{\mathscr{A}_0^2 (\overline{B}_{\rho})} +O(\varepsilon) \right)\alpha^2(\Omega). \notag
	\end{eqnarray}
\end{proof}

We are now in the position to establish the stability of the second type of weighted inequality in $\mathbb{H}^{n+1}$, which states that for any fixed $0\leqslant k\leqslant n$, $-1\leqslant j<k$, if $\mathscr{A}_j(\Omega)=\mathscr{A}_j(\overline{B}_{\rho})$ holds, then
\begin{equation*}
	\int_M \phi'\sigma_k(\kappa)\mathrm{d}\mu_g - \int_{\partial \overline{B}_{\rho}}\phi'\sigma_k(\kappa) \mathrm{d}\mu_g \geqslant C\alpha^2(\Omega),
\end{equation*}
where $C>0$ is a constant independent of $\Omega$. This is the quantitative version of (\ref{weightedH1}).

\begin{proof}[Proof of \eqref{hqq2} in Theorem \ref{mainresult3}]
	Substituting (\ref{volumeq}) (when $j=-1$) and (\ref{ujjj}) (when $0\leqslant j <k$) in (\ref{phi'sigmak}), using (\ref{hd}) we get
	\begin{eqnarray}
		&&\int_M \phi'\sigma_k(\kappa) \mathrm{d}\mu_g - \int_{\partial \overline{B}_{\rho}} \phi'\sigma_k(\kappa) \mathrm{d}\mu_g\notag\\
		&=& \int_{\mathbb{S}^n} {n\choose k} \left[ \dfrac{(n-k)(j-k)}{2} \phi^{n-k-2}(\rho)\phi'^{k+1}(\rho) - K\dfrac{1+n}{2} \phi^{n-k}(\rho)\phi'^{k+1}(\rho) \right.\notag\\
		&&\ \ \ \ \ \ \ \ \ \ \ \ \left. + K \dfrac{(k+1)(k-j-1)}{2} \phi^{n-k}(\rho)\phi'^{k-1}(\rho) \right] \rho^2 u^2 \mathrm{d}A \notag\\
		&&+ \int_{\mathbb{S}^n} {n\choose k} \left[ \dfrac{(n-k)(k-j)}{2n}\phi^{n-k-2}(\rho)\phi'^{k+1}(\rho) \right.\notag\\
		&&\ \ \ \ \ \ \ \ \ \ \ \ \ \ \left.- K\dfrac{(k+1)(k-j-1)}{2n} \phi^{n-k}(\rho) \phi'^{k-1}(\rho) \right] \rho^2 |\nabla u|^2 \mathrm{d}A \notag\\
		&&+O(\varepsilon)\|u\|_{L^2(\mathbb{S}^n)}^2 + O(\varepsilon)\|\nabla u\|_{L^2(\mathbb{S}^n)}^2 \notag\\
		&=&C_3(n,k,j,\rho) \|u\|^2_{L^2(\mathbb{S}^n)} + C_4(n,k,j,\rho) \|\nabla u\|^2_{L^2(\mathbb{S}^n)} \notag\\
		&&+ O(\varepsilon) \|u\|_{L^2(\mathbb{S}^n)}^2 + O(\varepsilon) \|\nabla u\|_{L^2(\mathbb{S}^n)}^2.  \label{phi'sig}
	\end{eqnarray}
	
	Note that when $K=-1$, $C_4(n,k,j,\rho)>0$, thus by the Poincar\'e-type estimate (\ref{egenest}) and (\ref{egenest22}), we have
	\begin{eqnarray}
		&&\int_M \phi'\sigma_k(\kappa) \mathrm{d}\mu_g - \int_{\partial \overline{B}_{\rho}} \phi'\sigma_k(\kappa) \mathrm{d}\mu_g \notag \\
		&\geqslant& (C_3(n,k,j,\rho) + C_4(n,k,j,\rho)n) \|u\|_{L^2(\mathbb{S}^n)}^2 + \dfrac{1}{2}C_4(n,k,j,\rho) \|\nabla u\|_{L^2(\mathbb{S}^n)}^2 \notag\\
		&&+O(\varepsilon)\|u\|_{L^2(\mathbb{S}^n)}^2 + O(\varepsilon) \|\nabla u\|_{L^2(\mathbb{S}^n)}^2 \notag\\
		&\geqslant& \dfrac{1}{2} C_4(n,k,j,\rho) \|\nabla u\|_{L^2(\mathbb{S}^n)}^2 \notag\\
		&\geqslant& {n\choose k}\dfrac{(n-k)(k-j)}{4n} \phi^{n-k-2}(\rho)\phi'^{k+1}(\rho) \rho^2 \|\nabla u\|_{L^2(\mathbb{S}^n)}^2, \label{phi'est}
	\end{eqnarray}
	where we used the fact that $C_3(n,k,j,\rho)+C_4(n,k,j,\rho)n\geqslant 0$ in the second inequality by a direct computation.
	
	Finally, combining (\ref{phi'est}) with the estimate (\ref{nablau}), we conclude that
	\begin{eqnarray}
		&&\int_M \phi'\sigma_k(\kappa) \mathrm{d}\mu_g - \int_{\partial \overline{B}_{\rho}} \phi'\sigma_k(\kappa) \mathrm{d}\mu_g \notag \\
		&\geqslant& \left({n\choose k} \dfrac{n(n-k)(k-j)}{4} \dfrac{\phi'^{k+1}(\rho)}{\mathrm{Area}(\mathbb{S}^n)\phi^{n+k+2}(\rho)} + O(\varepsilon) \right)\alpha^2(\Omega).  \notag
		%&\geqslant& \left( \dfrac{n(n-k)(k-j)}{4\phi^2(\rho)}\dfrac{{n\choose k} \phi^{n-k}(\rho)\phi'^{k+1}(\rho)\mathrm{Area}(\mathbb{S}^n)}{\mathscr{A}_0^2(\overline{B}_{\rho})} +O(\varepsilon)\right)\alpha^2(\Omega). \notag
	\end{eqnarray}
\end{proof}

We remark that the quantitative weighted quermassintegral inequalities in $\mathbb{S}^{n+1}$ can also be proved using this method under the same condition as in Theorem \ref{mainresult3} for some special $j<k$. However, we can't confirm that the coefficient of $\displaystyle\int_{\mathbb{S}^n} |\nabla u|^2 \mathrm{d}A$  is positive for all $-1\leqslant j<k\ (0\leqslant k\leqslant n)$ in the approximate expression of the curvature integral deficit. Though we believe the quantitative weighted quermassintegral inequalities hold for all $j<k$ in $\mathbb{S}^{n+1}$.

\begin{ack}
The authors would like to thank Professor Yong Wei for his helpful discussions and constant support. The authors were supported by National Key Research and Development Program of China 2021YFA1001800.
\end{ack}
%----------------------------------------------------------

\bibliographystyle{plainnat}
\bibliography{reference.bib}

\begin{thebibliography}{24}
\providecommand{\natexlab}[1]{#1}
\providecommand{\url}[1]{\texttt{#1}}
\expandafter\ifx\csname urlstyle\endcsname\relax
  \providecommand{\doi}[1]{doi: #1}\else
  \providecommand{\doi}{doi: \begingroup \urlstyle{rm}\Url}\fi

\bibitem[B{\"o}gelein et~al.(2015)B{\"o}gelein, Duzaar, and
  Scheven]{Bgelein2015ASQ}
Verena B{\"o}gelein, Frank Duzaar, and Christoph Scheven.
\newblock A sharp quantitative isoperimetric inequality in hyperbolic
  $n$-space.
\newblock \emph{Calculus of Variations and Partial Differential Equations},
  54:\penalty0 3967--4017, 2015.

\bibitem[B{\"o}gelein et~al.(2016)B{\"o}gelein, Duzaar, and
  Fusco]{Bgelein2016AQI}
Verena B{\"o}gelein, Frank Duzaar, and Nicola Fusco.
\newblock A quantitative isoperimetric inequality on the sphere.
\newblock \emph{Advances in Calculus of Variations}, 10\penalty0 (3):\penalty0
  223--265, 2016.

\bibitem[Brendle et~al.(2016)Brendle, Hung, and Wang]{B-H-W}
Simon Brendle, Pei-Ken Hung, and Mu-Tao Wang.
\newblock A {Minkowski} inequality for hypersurfaces in the {Anti}-de
  {Sitter-Schwarzschild} manifold.
\newblock \emph{Communications on Pure and Applied Mathematics}, 69\penalty0
  (1):\penalty0 124--144, 2016.

\bibitem[Brendle et~al.(2018)Brendle, Guan, and Li]{B-G-L}
Simon Brendle, Pengfei Guan, and Junfang Li.
\newblock An inverse curvature type hypersurface flow in {$\mathbb{H}^{n+1}$}.
\newblock \emph{preprint}, 2018.

\bibitem[Cicalese and Leonardi(2012)]{Cicalese2012ASP}
Marco Cicalese and Gian~Paolo Leonardi.
\newblock A selection principle for the sharp quantitative isoperimetric
  inequality.
\newblock \emph{Archive for Rational Mechanics and Analysis}, 206:\penalty0
  617--643, 2012.

\bibitem[de~Lima and Gir{\~a}o(2016)]{Lima-Girao}
Levi~Lopes de~Lima and Frederico Gir{\~a}o.
\newblock An {Alexandrov–Fenchel-Type} inequality in hyperbolic space with an
  application to a penrose inequality.
\newblock \emph{Annales Henri Poincar{\'e}}, 17:\penalty0 979--1002, 2016.

\bibitem[Fuglede(1986)]{Fuglede1986StabilityIT}
Bent Fuglede.
\newblock Stability in the isoperimetric problem.
\newblock \emph{Bulletin of The London Mathematical Society}, 18:\penalty0
  599--605, 1986.

\bibitem[Fuglede(1989)]{Fuglede1989StabilityIT}
Bent Fuglede.
\newblock Stability in the isoperimetric problem for convex or nearly spherical
  domains in $\mathbb{R}^n$.
\newblock \emph{Transactions of the American Mathematical Society},
  314\penalty0 (2):\penalty0 619--638, 1989.

\bibitem[Fusco and {La Manna}(2023)]{FUSCO2023109946}
Nicola Fusco and Domenico~Angelo {La Manna}.
\newblock Some weighted isoperimetric inequalities in quantitative form.
\newblock \emph{Journal of Functional Analysis}, 285\penalty0 (2):\penalty0
  109946, 2023.

\bibitem[Gavitone et~al.(2020)Gavitone, Manna, Paoli, and Trani]{Gavitone}
Nunzia Gavitone, Domenico Angelo~La Manna, Gloria Paoli, and Leonardo Trani.
\newblock A quantitative {Weinstock} inequality for convex sets.
\newblock \emph{Calculus of Variations and Partial Differential Equations},
  59\penalty0 (2), 2020.

\bibitem[Ge et~al.(2015)Ge, Wang, and Wu]{Ge-Wang-Wu}
Yuxin Ge, Guofang Wang, and Jie Wu.
\newblock The {GBC} mass for asymptotically hyperbolic manifolds.
\newblock \emph{Mathematische Zeitschrift}, 281:\penalty0 257--297, 2015.

\bibitem[Gir{\~a}o and Rodrigues(2020)]{Girao-Rodrigues}
Frederico Gir{\~a}o and Diego Rodrigues.
\newblock Weighted geometric inequalities for hypersurfaces in sub-static
  manifolds.
\newblock \emph{Bulletin of the London Mathematical Society}, 52\penalty0
  (5):\penalty0 121--136, 2020.

\bibitem[Glaudo(2022)]{Glaudo2021MinkowskiIF}
Federico Glaudo.
\newblock Minkowski inequality for nearly spherical domains.
\newblock \emph{Advances in Mathematics}, 408:\penalty0 108595, 2022.

\bibitem[Guan and Li(2009)]{GL09}
P.~Guan and J.~Li.
\newblock The quermassintegral inequalities for k-convex starshaped domains.
\newblock \emph{Adv. Math.}, 221:\penalty0 1725--1732, 2009.

\bibitem[Guan(2014)]{Guan2014}
Pengfei Guan.
\newblock \emph{Curvature Measures, Isoperimetric Type Inequalities and Fully
  Nonlinear PDEs}, pages 47--94.
\newblock Springer International Publishing, Cham, 2014.

\bibitem[Guan and Li(2015)]{Guan2013AMC}
Pengfei Guan and Junfang Li.
\newblock A mean curvature type flow in space forms.
\newblock \emph{International Mathematics Research Notices}, 2015\penalty0
  (13):\penalty0 4716--4740, 2015.

\bibitem[Guan and Li(2021)]{Li2021IsoperimetricTI}
Pengfei Guan and Junfang Li.
\newblock Isoperimetric type inequalities and hypersurface flows.
\newblock \emph{The Journal of Men's Studies}, 54:\penalty0 56--80, 2021.

\bibitem[Hu et~al.(2022)Hu, Li, and Wei]{Locallyconstrained}
Yingxiang Hu, Haizhong Li, and Yong Wei.
\newblock Locally constrained curvature flows and geometric inequalities in
  hyperbolic space.
\newblock \emph{Mathematische Annalen}, 382:\penalty0 1425--1474, 2022.

\bibitem[Kwong and Miao(2015)]{Kwong-Miao15}
Kwok-Kun Kwong and Pengzi Miao.
\newblock Monotone quantities involving a weighted $\sigma_k$ integral along
  inverse curvature flows.
\newblock \emph{Communications in Contemporary Mathematics}, 17\penalty0
  (5):\penalty0 1550014, 2015.

\bibitem[Scheuer(2023)]{Scheuer2023Stability}
Julian Scheuer.
\newblock Stability from rigidity via umbilicity.
\newblock \emph{arXiv:2103.07178v2}, 2023.

\bibitem[Solanes(2005)]{G-Solanes05}
Gil Solanes.
\newblock Integral geometry and the {Gauss-Bonnet} theorem in constant
  curvature spaces.
\newblock \emph{Transactions of the American Mathematical Society},
  358:\penalty0 1105--1115, 2005.

\bibitem[VanBlargan and Wang(2022{\natexlab{a}})]{VanBlargan2022QuantitativeQI}
Caroline VanBlargan and Yi~Wang.
\newblock Quantitative quermassintegral inequalities for nearly spherical sets.
\newblock \emph{arXiv:2201.04256}, 2022{\natexlab{a}}.

\bibitem[VanBlargan and Wang(2022{\natexlab{b}})]{VanBlargan2022StabilityOQ}
Caroline VanBlargan and Yi~Wang.
\newblock Stability of quermassintegral inequalities along inverse curvature
  flows.
\newblock \emph{arXiv:2208.14341}, 2022{\natexlab{b}}.

\bibitem[Wei and Zhou(2023)]{Wei2022NewWG}
Yong Wei and Tailong Zhou.
\newblock New weighted geometric inequalities for hypersurfaces in space forms.
\newblock \emph{Bulletin of the London Mathematical Society}, 55\penalty0
  (1):\penalty0 263--281, 2023.

\end{thebibliography}

\end{document}